\documentclass[12pt, twoside]{article}
\usepackage{amsmath,amsthm,amssymb}
\usepackage{times}
\usepackage{enumerate}

\usepackage[curve,matrix,arrow,frame,tips]{xy}

\usepackage[usenames,dvipsnames]{color}

\usepackage[normalem]{ulem}

\def\titlerunning#1{\gdef\titrun{#1}}
\makeatletter
\def\author#1{\gdef\autrun{\def\and{\unskip, }#1}\gdef\@author{#1}}
\def\address#1{{\def\and{\\\hspace*{18pt}}\renewcommand{\thefootnote}{}%
\footnote {#1}}%
\markboth{\autrun}{\titrun}}
\makeatother
\def\email#1{e-mail: #1}
\def\subjclass#1{{\renewcommand{\thefootnote}{}%
\footnote{\emph{Mathematics Subject Classification (2010):} #1}}}

\newtheorem{theorem}{Theorem}[section]
\newtheorem{corollary}[theorem]{Corollary}
\newtheorem{lemma}[theorem]{Lemma}

\newtheorem{proposition}[theorem]{Proposition}
\newtheorem{hypothesis}[theorem]{Hypothesis}
\newtheorem*{thmA}{Theorem A}
\newtheorem*{thmB}{Theorem B}
\newtheorem*{thmC}{Theorem C}
\newtheorem*{thmD}{Theorem D}

\theoremstyle{definition}
\newtheorem{definition}[theorem]{Definition}
\newtheorem{remark}[theorem]{Remark}

\numberwithin{equation}{section}

\newcommand{\mr}[1]{\mathrm{#1}}
\newcommand{\mf}[1]{\mathfrak{#1}}
\newcommand{\mc}[1]{\mathcal{#1}}
\newcommand{\mb}[1]{\mathbb{#1}}
\newcommand{\mbf}[1]{\mathbf{#1}}

\newcommand{\Z}{\mb{Z}}
\newcommand{\Q}{\mb{Q}}
\newcommand{\zp}{\mb{Z}_p}
\newcommand{\qp}{\mb{Q}_p}

\newcommand{\F}{\mb{F}}

\newcommand{\RGa}{\mathbf{R}\Gamma}
\newcommand{\La}{\Lambda}
\newcommand{\RHom}{\mbf{R}\!\Hom}

\newcommand{\Iw}{\mr{Iw}}

\newcommand{\ab}{\mr{\mr{ab}}}

\newcommand{\rQ}{\mr{Q}}
\newcommand{\rF}{\mr{F}}
\newcommand{\rC}{\mr{C}}

\newcommand{\rH}{{\rm H}}
\newcommand{\rE}{{\rm E}}
\newcommand{\rT}{\mr{T}}

\newcommand{\ps}[1]{[\![ #1 ]\!]}

\newcommand{\cotimes}[1]{\,\hat{\otimes}_{#1} \,}

\newcommand{\tf}{\mathrm{tf}}

\DeclareMathOperator{\Hom}{Hom} 
 
\DeclareMathOperator{\Gal}{Gal}
\DeclareMathOperator{\Res}{Res}
\DeclareMathOperator{\coker}{coker} 
\DeclareMathOperator{\Cone}{Cone}

\DeclareMathOperator{\im}{im}

\DeclareMathOperator{\Inf}{Inf}
\DeclareMathOperator{\Ext}{Ext}
\DeclareMathOperator{\rank}{rank}
\DeclareMathOperator{\Fitt}{Fitt}
\DeclareMathOperator{\Tor}{Tor}
\DeclareMathOperator{\Tra}{Tra}

\begin{document}


\baselineskip=17pt


\titlerunning{Exterior Powers in Iwasawa Theory}

\title{Exterior Powers in Iwasawa Theory}

\author{F. M. Bleher \and T. Chinburg \and R. Greenberg \and M. Kakde \and R. Sharifi \and M. J. Taylor}

\date{}

\maketitle

\address{F. M. Bleher:  Dept. of Mathematics, University of Iowa, Iowa City, IA 52242, USA; \email{frauke-bleher@uiowa.edu}
\and
T. Chinburg: Dept. of Mathematics, University of Pennsylvania, Philadelphia, PA 19104, USA; \email{ted@math.upenn.edu}
\and
R. Greenberg: Dept. of Mathematics, University of Washington, Seattle, WA 98195, USA; \email{greenber@math.washington.edu}
\and
M. Kakde: Dept. of Mathematics, Indian Institute of Science, Bangalore 560012, India; \email{maheshkakde@iisc.ac.in}
\and
R. Sharifi: Dept. of Mathematics, University of California, Los Angeles, Los Angeles, CA 90095, USA; \email{sharifi@math.ucla.edu}
\and
M. J. Taylor: School of Mathematics, Merton College, University of Oxford, Oxford OX1 4JD, UK; 
\email{martin.taylor@merton.ox.ac.uk}
}

\subjclass{Primary 11R23; Secondary 11R34}


\begin{abstract}
	The Iwasawa theory of CM fields has traditionally concerned Iwasawa modules that are  abelian pro-$p$ Galois groups 
	with ramification allowed at a maximal set of primes over $p$ such that the module is torsion.  A main conjecture for such an 
	Iwasawa module describes its codimension one support in terms of a $p$-adic $L$-function attached to the primes of ramification. 
	In this paper, we study more general and potentially much smaller  modules that are quotients of exterior powers of 
	Iwasawa modules with ramification at a set of primes over $p$ by sums of exterior powers of inertia subgroups.  We show that the 
	higher codimension support of such quotients can be measured by finite collections of characteristic ideals of classical Iwasawa 
	modules, hence by $p$-adic $L$-functions under the relevant CM main conjectures.
\end{abstract}

\section{Introduction}
\label{s:intro}

Iwasawa theory studies the growth of Selmer groups in towers of number fields.  In the commutative setting, these towers have Galois groups isomorphic to $\zp^r$ for some $r \ge 1$, and their Iwasawa algebras are isomorphic to a power series ring in $r$ variables over $\zp$.  The Selmer groups are typically attached to Galois-stable lattices in $p$-adic Galois representations that come from geometry.  The local conditions defining the Selmer groups are chosen so that the Pontryagin dual of a limit up the tower is a finitely generated torsion module over the Iwasawa algebra.  For example, when the Galois representation is the trivial representation, these dual Selmer groups are abelian pro-$p$ Galois groups with restricted ramification.  In many instances, one can construct a power series that gives rise to a $p$-adic $L$-function attached to the lattice and the Selmer conditions.  In what is known as a main conjecture, this power series is conjectured to generate the characteristic ideal of the Iwasawa module.

In this paper, we develop a method to study the support of Iwasawa modules in arbitrary codimension, focusing specifically on the Iwasawa theory of CM fields for one-dimensional Galois representations.  To study the codimension $n$ support of a finitely generated Iwasawa module, we use the $n$th Chern class of its maximal codimension $n$ submodule.  This Chern class, as defined in \cite{BCGKPST}, is the sum of the lengths of its localizations at the prime ideals of codimension $n$.  For instance, the first Chern class of a finitely generated torsion Iwasawa module is the divisor defining its characteristic ideal.

A CM main conjecture describes the first Chern class of an Iwasawa module unramified outside of a ($p$-adic) CM type of primes over $p$ in terms of a Katz $p$-adic $L$-function.  Recall that a CM type is a set of one from each pair of complex conjugate primes over $p$ in a CM field, supposing that the primes over $p$ split from the maximal totally real subfield. We aim to construct an Iwasawa module which has support in higher codimension related to a tuple of $p$-adic $L$-functions for distinct CM types. For this, we take the quotient of the top exterior power of a $p$-ramified Iwasawa module by a sum of top exterior powers of composites of inertia groups at certain of the primes.   The main results of this paper relate higher Chern classes of these \emph{exterior quotients} to the first Chern classes of Iwasawa modules unramified outside of a CM type, and therefore to Katz $p$-adic $L$-functions if the relevant CM main conjectures hold.  

The idea of taking top exterior powers occurs frequently in number theory, as characteristic ideals arise as determinants.  The quotient of the top exterior power of a finitely generated free module by the top exterior power of a free submodule of full rank has first Chern class equal to that of the quotient of the two free modules.  For this reason, exterior powers figure heavily in equivariant formulations of main conjectures using determinants, as in the work of Fukaya and Kato \cite{FK}.  They also appear prominently in Stark's conjectures, in which one considers the top exterior powers of isotypic components of unit groups in order to arrive at regulators which are related to the special values of derivatives of Artin $L$-series.  Our work has the seemingly unique aspect that we take a quotient of a top exterior power of an Iwasawa module by a sum of two or more top exterior powers of submodules.

Let us briefly describe our main theorems, as we shall state after introducting the necessary framework.  Theorem A relates the codimension $2$ support of an exterior quotient to a pair of first Chern classes corresponding to arbitrary distinct choices of CM types. In Theorem B, by localizing away from bad primes, we obtain an isomorphism between an exterior quotient and the quotient of an Iwasawa algebra by the ideal generated by a tuple of first Chern classes.  Theorem C involves two CM types differing in a  degree one prime, in which case our quotient is the classical Iwasawa module unramified outside the intersection of the two CM types.  We relate the sum of second Chern classes of this module and another for the complex conjugate set to the ideal generated by the two first Chern classes of the CM types.  Finally, in Theorem D, we describe a quotient of second exterior powers as a Galois group with restricted ramification.

We turn to details of our work, starting with the formal definition of our key invariant.  An index of notations is given in Section \ref{s:NoteIndex} at the end of the paper.  For a finitely generated Iwasawa module $M$, we let $t_n(M)$ denote the $n$th Chern class of the maximal 
submodule $\rT_n(M)$ of $M$ supported in codimension at least $n$. 
That is, $t_n(M)$ is the formal sum
$$
	t_n(M) = \sum_{\mc{P}} \mr{length}(\rT_n(M)_\mc{P})[\mc{P}]
$$
over height $n$ prime ideals $\mc{P}$ in the Iwasawa algebra.  In the case that $M = \rT_n(M)$, this is the $n$th Chern class $c_n(M)$ of $M$
considered in \cite{BCGKPST}.  The invariant $t_1(M)$ is naturally identified with the characteristic ideal of the torsion submodule of $M$, matching the classical definition.
Note that $t_n$ is not additive on arbitrary exact sequences of finitely generated modules, but it is on exact sequences of modules supported in codimension at least $n$.

Now, let $p$ be an odd prime, and let $E$ be a CM field of degree $2d$.  We suppose that each prime over $p$ in the maximal totally real subfield $E^+$ of $E$ splits in $E$. Let $F$ be a finite abelian extension of $E$ of degree prime to $p$ containing the $p$th roots of unity. Let $K$ be the compositum of $F$ with all of the $\zp$-extensions of $E$, and let $\Gamma=\Gal(K/F)$ and $\mc{G}=\Gal(K/E)$.  Let $\Sigma$ be a subset of the set of primes of $E$ over $p$.  
We consider the $\Sigma$-ramified Iwasawa module $X_{\Sigma}$ that is the Galois group over $K$ of the maximal unramified outside of $\Sigma$ abelian pro-$p$ extension of $K$.
 Then $\Gamma$ is isomorphic to $\zp^r$ for some integer $r\ge d+1$, where $r = d+1$ if the Leopoldt conjecture is true.  
Let 
$$
	\psi \colon \Delta = \Gal(F/E) \to W^{\times}
$$ 
be a $p$-adic character, where $W$ denotes the Witt vectors of an algebraic closure $\overline{\mathbb{F}}_p$ of $\mathbb{F}_p$.  (In our main results, $W$ may be replaced by the ring generated by the values of $\psi$.)
Let $\Lambda = W\ps{\Gamma}$ be the completed group ring of $\Gamma$ over $W$, which is a power series ring in $r$ variables over $W$.  
We are interested in the finitely generated $\Lambda$-module $X_{\Sigma}^{\psi} = X_{\Sigma} \cotimes{\zp[\Delta]} W$ for the map $\zp[\Delta] \to W$ induced by $\psi$, which is to say the $\psi$-isotypical component of $X_{\Sigma}$, or more precisely of its completed tensor product with $W$.  

Let $S_f$ be the set of all primes over $p$ in $E$.
A ($p$-adic) CM type $\Sigma$ is a subset of $S_f$ which contains exactly one prime of each conjugate pair.  
One has a power series $\mr{L}_{\Sigma,\psi} \in \Lambda$ that gives rise to a certain Katz $p$-adic $L$-function
attached to $\Sigma$ and $\psi$.  Hida and Tilouine \cite{HT} showed that $X_{\Sigma}^{\psi}$ is $\Lambda$-torsion and stated an Iwasawa main conjecture that says that the characteristic ideal of $X_{\Sigma}^{\psi}$ is generated by $\mr{L}_{\Sigma,\psi}$.  They proved an anticyclotomic variant of this conjecture under certain hypotheses.  Work of Hsieh \cite{Hsieh} shows that the characteristic ideal of $X_{\Sigma}^{\psi}$ is divisible by $\mr{L}_{\Sigma,\psi}$ under certain assumptions.  In particular, this relates the codimension one support of the algebraically defined module $X_{\Sigma}^{\psi}$ to that of the analytically defined module $\Lambda/(\mr{L}_{\Sigma,\psi})$.
We will use $\mc{L}_{\Sigma,\psi}$ to denote a choice of generator of the characteristic ideal of $X_{\Sigma}^{\psi}$.  The CM main conjecture for $\Sigma$ is then the statement that $(\mc{L}_{\Sigma,\psi}) = (\mr{L}_{\Sigma,\psi})$.

Fix a set $\mc{S}$ of primes over $p$ properly containing a CM type. 
Let us write $\mc{S}$ as a union of two distinct CM types $\mc{S}_1$ and $\mc{S}_2$.
Let $\theta$ be a greatest common divisor in $\Lambda$ of 
$\mc{L}_{\mc{S}_1,\psi}$ and $\mc{L}_{{\mc{S}_2},\psi}$. 
For a discussion of a possible construction of examples in which $\theta$ is a non-unit, see Remark \ref{rem:divisorrem}.
The first Chern class of the quotient
$\Lambda/(\mc{L}_{\mc{S}_1,\psi},\mc{L}_{{\mc{S}_2},\psi})$
is the ideal $\Lambda \theta$.  Our interest in this paper is the more subtle information 
contained in the pseudo-null module 
\begin{equation}
\label{eq:bigenchilada}
\rT_2\left ( \frac{\Lambda}{(\mc{L}_{\mc{S}_1,\psi},\mc{L}_{{\mc{S}_2},\psi})}\right )  = \frac{\Lambda\theta}{(\mc{L}_{\mc{S}_1,\psi},\mc{L}_{{\mc{S}_2},\psi})} \cong   \frac{\Lambda}{(\mc{L}_{\mc{S}_1,\psi}/
\theta,\mc{L}_{{\mc{S}_2},\psi}/\theta)}.
\end{equation}

We aim to relate the codimension two support of the module (\ref{eq:bigenchilada}) to 
that of some naturally defined algebraic modules, as was done in \cite{BCGKPST} for imaginary
quadratic fields $E$ under the assumption of coprimality of $\mc{L}_{\mc{S}_1,\psi}$ and $\mc{L}_{{\mc{S}_2},\psi}$.  
This requires overcoming a serious obstruction for $E$ an arbitrary CM field.  Namely,
 the $\Lambda$-rank $\ell$ of $X_{\mc{S}}^{\psi}$ may now be larger than $1$: that is, we show in Lemma \ref{lem:sweepup} that
$$
	\ell = \sum_{v \in \mc{S} - \Sigma} [E_v:\mathbb{Q}_p],
$$ 
where $\Sigma$ is any CM type contained in $\mc{S}$.
If $\ell > 1$, then the first Chern class of $X_{\mc{S}_i}^{\psi}$ for $i \in \{1,2\}$ is insufficient to identify, up to errors supported in codimension greater than $2$, the $\Lambda$-submodule
$I_{\mc{T}_i}^{\psi}$ of $X_{\mc{S}}^{\psi}$ generated by inertia groups at primes over $\mc{T}_i = \mc{S} - \mc{S}_i$.

We make the simple but key observation that the $\ell$th exterior powers 
of $X_{\mc{S}}^{\psi}$ and the $I_{\mc{T}_i}^{\psi}$ are indeed rank one $\Lambda$-modules.  We therefore replace the quotient
$X_{\mc{S}}^{\psi}/(I_{\mc{T}_1}^{\psi} + I_{\mc{T}_2}^{\psi}) \cong X_{\mc{S}_1 \cap \mc{S}_2}^{\psi}$
found in the imaginary quadratic setting by the exterior quotient
\begin{equation} \label{eq:extquot}
    \frac{(\bigwedge^\ell X_{\mc{S}}^{\psi})_{\tf}}{(\bigwedge^\ell I_{\mc{T}_1}^{\psi})_{\tf}  
    + (\bigwedge^\ell I_{\mc{T}_2}^{\psi})_{\tf}}
\end{equation}
where a subscript ``$\mathrm{tf}$'' denotes maximal $\Lambda$-torsion-free quotient. Here, 
we view each $(\bigwedge^\ell I_{\mc{T}_i}^{\psi})_{\tf}$ as a submodule of $(\bigwedge^\ell X_{\mc{S}}^{\psi})_{\tf}$ and
take their sum within the latter group.
We will compare the second Chern classes of the maximal pseudo-null submodules 
of (\ref{eq:extquot}) and of $\Lambda/(\mc{L}_{\mc{S}_1,\psi},\mc{L}_{{\mc{S}_2},\psi})$.

For a compact $\Lambda$-module $A$, 
we let $A(1)$ be the Tate twist of $A$ by the cyclotomic character of $\Gamma$.
Let $A^{\iota}$ denote the $\Lambda$-module which as a topological $\zp$-module is $A$ 
and on which $\gamma \in \Gamma$ now acts by $\gamma^{-1}$. For $A$ finitely generated, we define 
$$
	\rE^i(A) = \Ext^i_{\Lambda}(A^{\iota},\Lambda),
$$ 
we set $A_{\tf} = A/\rT_1(A)$, we let $\bigwedge^{\ell} A$ denote the $\ell$th exterior power of $A$ over $\Lambda$, and we let 
$\Fitt(A)$ denote the $0$th Fitting ideal of $A$.

Write $\mc{S}^c$ for the set of primes over $p$ not in $\mc{S}$.
Then $X_{\mc{S}^c}^{\omega\psi^{-1}}$ is a torsion $\Lambda$-module because $\mc{S}^c$
is contained in a CM type of primes over $p$.      
To simplify statements of our main theorems as stated in the body of this paper, we suppose in this introduction that $\psi$ (resp. $\omega\psi^{-1}$) is nontrivial on all decomposition groups in $\Delta$ at primes $\mf{p}\in \overline{\mc{S}}$ (resp. $\mf{p}\in \mc{S}$),
for $\overline{\mc{S}}$ the complex conjugate set to $\mc{S}$.
Under this assumption, each $I_{\mc{T}_i}^{\psi}$ is $\Lambda$-free, so each $\bigwedge^{\ell} I_{\mc{T}_i}^{\psi} = (\bigwedge^\ell I_{\mc{T}_i}^{\psi})_{\tf}$ is free of rank one. (The latter comment applies to the theorems in this introduction, so we omit the ``$\tf$'' notation on such groups in them.)

\begin{thmA} \label{thm:simple}
    For a union $\mc{S}$ of two distinct CM types $\mc{S}_1$ and $\mc{S}_2$ and its complement $\mc{S}^c$, 
    we have an equality of second Chern classes
    \begin{equation}
    \label{eq:c2formulasimple}
    t_2\left(\frac{\Lambda }{(\mc{L}_{\mc{S}_1,\psi},\mc{L}_{\mc{S}_2,\psi})} \right)  = 
    t_2\left(\frac{(\bigwedge^\ell X_{\mc{S}}^{\psi})_{\tf}}{\bigwedge^\ell I_{\mc{T}_1}^{\psi}  
    + \bigwedge^\ell I_{\mc{T}_2}^{\psi} }  \right ) 
    + t_2\left(\frac{\theta}{\theta_0} \cdot \frac{\Lambda}{\Fitt(\rE^2(X_{\mc{S}^c}^{\omega\psi^{-1}})(1))}\right)
    \end{equation}
    where $\ell = \rank_{\Lambda} X_{\mc{S}}^{\psi}$, where $\theta$ is a gcd of the characteristic elements $\mc{L}_{\mc{S}_i,\psi}$ of
    $X_{\mc{S}_i}^{\psi}$ for $i \in \{1,2\}$, and where $\theta_0$ is a generator of  
    $t_1(\bigwedge^{\ell} X_{\mc{S}}^{\psi})$.  
\end{thmA}

\begin{remark}
	In Theorems \ref{thm:unconditional} and \ref{thm:unconditionalBIG}, we generalize Theorem A to treat 
	$n$-tuples of CM types, without any assumption on $\psi$.
\end{remark}

The $\Lambda$-module $X_{\mc{S}^c}^{\omega\psi^{-1}}$ is a quotient of $X_{\Sigma}^{\omega\psi^{-1}}$ for each of the $2^{\ell}$ 
CM types $\Sigma$ containing $\mc{S}^c$, each of which has first Chern class 
$(\mc{L}_{\Sigma,\omega\psi^{-1}})$, and these lack obvious dependencies in general.  When $\ell > 1$, we therefore suspect that the $\Lambda$-module $X_{\mc{S}^c}^{\omega\psi^{-1}}$  frequently has annihilator of height greater than $2$, in which case the last term in \eqref{eq:c2formulasimple} vanishes.   (Recall that for a Cohen-Macaulay ring $R$, the height of the annihilator of a finitely generated $R$-module $M$ is at most the smallest
$i$ such that $\Ext^i(M,R)$ is nonzero \cite[Theorem 17.4]{matsumura}.)
In fact, the proof of Theorem A and a spectral sequence argument lead to the following.

\begin{thmB}
\label{thm:simplen}
	Let $\mc{S}$ be  a subset of $S_f$ that properly contains a CM type.  Let $\mf{q}$ be a prime of $\Lambda$ not in the support of
	$(X_{\mc{S}^c}^{\omega\psi^{-1}})^{\iota}(1)$.
	Then the following hold.
	\begin{enumerate}
		\item[(i)] The $\Lambda_{\mf{q}}$-module $X_{\mc{S},\mf{q}}^{\psi}$ is free of rank $\ell$.  
		In particular, $(\bigwedge^\ell X_{\mc{S},\mf{q}}^{\psi})_{\tf} = \bigwedge^\ell X_{\mc{S},\mf{q}}^{\psi}$.
		\item[(ii)] Let $\mc{S}_1, \ldots, \mc{S}_n$ be distinct CM types contained in $\mc{S}$ for some $n \ge 1$.  
		Then 
		$$
			\frac{\bigwedge^\ell X_{\mc{S},\mf{q}}^{\psi}}{\bigwedge^\ell I_{\mc{T}_1,\mf{q}}^{\psi}  
   	 		+ \cdots + \bigwedge^\ell I_{\mc{T}_n,\mf{q}}^{\psi} } \cong
			\frac{\Lambda_{\mf{q}}}{(\mc{L}_{\mc{S}_1,\psi},\ldots,\mc{L}_{\mc{S}_n,\psi})},
		$$
		where $\mc{T}_i = \mc{S}-\mc{S}_i$ for each $i$.
	\end{enumerate}
\end{thmB}

The rank $\ell$ of $X_{\mc{S}}^{\psi}$ equals $1$ if and only if $\mc{S}$ is a union of two CM types $\mc{S}_1$ and $\mc{S}_2$ that differ in a single completely split prime.  In this case, supposing that $\mc{L}_{\mc{S}_1,\psi}$ and $\mc{L}_{\mc{S}_2,\psi}$ are relatively prime, we prove the following remarkably clean refinement of Theorem A, which rests on proving that $X_{\mc{S}_1 \cap \mc{S}_2}^{\psi}$ and $X_{\overline{\mc{S}}_1 \cap \overline{\mc{S}}_2}^{\omega\psi^{-1}}$ are pseudo-null under this assumption.

\begin{thmC} \label{thm:simplerank1}
	Suppose that $\ell = 1$, and suppose that $\mc{L}_{\mc{S}_1,\psi}$ and $\mc{L}_{\mc{S}_2,\psi}$ are relatively prime.
	Then we have
	\begin{equation}
	\label{eq:simpler1formula}
    		c_2\left(\frac{\Lambda }{(\mc{L}_{\mc{S}_1,\psi},\mc{L}_{\mc{S}_2,\psi})} \right)  = 
    		c_2(X_{\mc{S}_1 \cap \mc{S}_2}^{\psi}) + c_2((X_{\overline{\mc{S}}_1 \cap \overline{\mc{S}}_2}^{\omega\psi^{-1}})^{\iota}(1)),
	\end{equation}
	where $\overline{\mc{S}}_i$ denotes the conjugate CM type to $\mc{S}_i$ for $i \in \{1,2\}$.
\end{thmC}

\begin{remark}
	Theorem C is a direct generalization of \cite[Theorem 5.2.5]{BCGKPST}, which treated the case that $E$ is imaginary quadratic. 
	That we could prove this result was far more surprising to us than it might seem: at the time of the writing of \cite{BCGKPST}, 
	the fact that $X_{S_f}^{\psi}$ has rank $[E^+:\Q]$ stood as a serious obstacle to a generalization to arbitrary CM fields. 
	
	While one can derive Theorem C itself through Theorem A (in particular, as $X_{\mc{S}}^{\psi}$ is torsion-free 
	when the torsion module $X_{\mc{S}^c}^{\omega\psi^{-1}}$ is pseudo-null), we give a finer and more subtle version without assumption 
	on $\psi$ and an entirely separate proof in Theorem \ref{thm:rank1}.
\end{remark}

We will show in Proposition \ref{prop:pn} that if $\ell = 1$, then $\mc{L}_{\mc{S}_1,\psi}$ and $\mc{L}_{\mc{S}_2,\psi}$ are relatively prime if 
and only if both $X_{\mc{S}_1 \cap \mc{S}_2}^\psi$ and $X_{\overline{\mc{S}}_1 \cap \overline{\mc{S}}_2}^{\omega\psi^{-1}}$ 
are pseudo-null.

\begin{remark}
	Let us elaborate on a comment made earlier.  One can ask about the relationship between $X_{\mc{S}_1 \cap \mc{S}_2}^{\psi}$ and 
    	$\Lambda /(\mc{L}_{\mc{S}_1,\psi},\mc{L}_{\mc{S}_2,\psi})$ when $\ell > 1$.
    	The maximal pseudo-null submodules of $X_{\mc{S}_1}^\psi$ and $X_{\mc{S}_2}^{\psi}$ are trivial.   Therefore,
    	$\mc{L}_{\mc{S}_1,\psi}$ and $\mc{L}_{\mc{S}_2,\psi}$ are annihilators of $X_{\mc{S}_1}^\psi$ and $X_{\mc{S}_2}^{\psi}$, respectively, 
	so they annihilate their common quotient $X_{\mc{S}_1 \cap \mc{S}_2}^{\psi}$. Consequently, any prime 
    	ideal in the support of $X_{\mc{S}_1 \cap \mc{S}_2}^{\psi}$ should contain both $\mc{L}_{\mc{S}_1,\psi}$ 
    	and $\mc{L}_{\mc{S}_2,\psi}$, and hence should be in the support of 
    	$\Lambda/(\mc{L}_{\mc{S}_1,\psi},\mc{L}_{\mc{S}_2,\psi})$. However, even under the simplifying 
    	assumption that $X_{\mc{S}}^\psi$ is a free $\Lambda$-module, the converse 
    	is unlikely to hold in general.  A prime ideal $\mc{P}$ of $\Lambda$ could be in the support of 
    	both $X_{\mc{S}}^\psi/I_{\mc{T}_1}^\psi$ and $X_{\mc{S}}^\psi/I_{\mc{T}_2}^\psi$ but fail to be in the 
    	support of $X_{\mc{S}}^\psi/(I_{\mc{T}_1}^\psi+I_{\mc{T}_2}^\psi)=X_{\mc{S}_1 \cap \mc{S}_2}^{\psi}$. 
    	For example, $\Lambda$-module bases for $I_{\mc{T}_1}^\psi$ and $I_{\mc{T}_2}^\psi$ 
    	(assuming they are free) could each be linearly dependent modulo $\mc{P}$, but their union might easily 
    	contain a linearly independent subset modulo $\mc{P}$. 
\end{remark}

When $\ell > 1$, it is natural to ask there is an interpretation of the first term on the right-hand
side of \eqref{eq:c2formulasimple} as the second Chern class of a suitable Galois group.  We provide such an interpretation in the case that $\ell = 2$.

\begin{definition}
\label{def:Zdef} Let $L$ be the maximal abelian pro-$p$ extension of $K$ that is unramified outside of $\mc{S} = \mc{S}_1 \cup \mc{S}_2$,
so that $X_{\mc{S}} = \mathrm{Gal}(L/K)$.  
Let $N$  be the maximal abelian pro-$p$ extension of $L$ unramified outside $\mc{S}$  with the following properties:
\begin{enumerate}
\item[(i)] $N$ is Galois over $E$, and $\Gal(N/L)$ is central in $\Gal(N/K)$, and 
\item[(ii)] the natural commutator pairing
$X_\mc{S} \times X_\mc{S} \to \Gal(N/L)$ is Hermitian with respect to the action of $\mc{G}=\Gal(K/E)$.  
\end{enumerate}
Let $M$ be the maximal subextension of $N$ containing $L$ such that $M/L$ is unramified outside $\mc{S}_1 \cap \mc{S}_2$.  Set $U = \Gal(N/L)$ and $V = \Gal(M/L)$.  
\end{definition}

We show that there is a canonical square root of the conjugation action of $\mc{G} = \Gal(K/E)$ on $U$ and on $V$; see Remark \ref{rem:result}. 
We consider the $\psi$-isotypical components $U^{\sqrt{\psi}}$ and $V^{\sqrt{\psi}}$ of $U$ and $V$, respectively, for this square root action.  
The $\psi^2$-isotypical component of the usual conjugation action of $\Delta$ on $V$ is the direct sum of $V^{\sqrt{\psi'\/}}$ over all characters $\psi'$ for which $\psi'^2= \psi^2$.  

\begin{thmD}  
\label{thm:simplerank2}
Suppose $\ell = 2$.  Let $\im(\Tor(U^{\sqrt{\psi}}))$ denote the image of \ $\Tor(U^{\sqrt{\psi}})$
in $V^{\sqrt {\psi}}$ under the homomorphism induced by the surjection $U \twoheadrightarrow V$.  The commutator pairing on $X_{\mc{S}}$ induces an isomorphism
$\bigwedge_{\Omega}^2 X_{\mc{S}}^\psi \xrightarrow{\sim} U^{\sqrt{\psi}}$ and surjections  
$$
 \frac{\bigwedge_{\Omega}^2 X_{\mc{S}}^\psi}{\bigwedge_{\Omega}^2  I_{\mc{T}_1}^\psi + \bigwedge_{\Omega}^2  I_{\mc{T}_2 }^\psi}
\twoheadrightarrow V^{\sqrt {\psi}} \quad \mbox{and} \quad  \frac{(\bigwedge_{\Omega}^2 X_{\mc{S}}^\psi)_{\tf}}{\bigwedge_{\Omega}^2  I_{\mc{T}_1}^\psi + \bigwedge_{\Omega}^2  I_{\mc{T}_2 }^\psi} \twoheadrightarrow \frac{V^{\sqrt{\psi}}}{\im(\Tor(U^{\sqrt{\psi}}))} 
$$
whose kernels are supported in codimension at least $3$. 	
\end{thmD}

Theorem D is proved in Theorem \ref{thm:mainquartic}. 
For a field diagram summarizing the groups and fields it involves, see Appendix \ref{s:FieldDiagram}.
The significance of this theorem is that when $\ell = 2$, a particular graded piece of a higher term in the lower central series of the Galois group of the maximal unramified outside $\mathcal{S}$ pro-$p$ extension $K_{\mathcal{S}}^{(p)}$ of $K$ arises when one seeks  a Galois-theoretic interpretation of natural modules defined by $p$-adic $L$-functions.  If $ V^{\sqrt{\psi}}$ is pseudo-null, one has
$$t_2\left (V^{\sqrt {\psi}} \right ) - t_2\left (\frac{V^{\sqrt{\psi}}}{\im(\Tor(U^{\sqrt{\psi}}))} \right ) = t_2 \left (\im(\Tor(U^{\sqrt{\psi}}))\right ).$$
However, $t_2$ is not an exact functor on exact sequences of modules that are not pseudo-null, and we do not know in general
whether $V^{\sqrt{\psi}}$ is pseudo-null.

\begin{remark}
It would be natural to consider how to generalize Theorem D for $\ell > 2$.  
If $N_1$ is the maximal abelian pro-$p$ unramified outside $\mathcal{S}$ extension of
the field $L$ in Definition \ref{def:Zdef}, then $X_{\mathcal{S}} = \mathrm{Gal}(L/K)$ and $T = \mathrm{Gal}(N_1/L)$
are the first two successive quotients in the derived series of $\mathrm{Gal}(K_{\mathcal{S}}^{(p)}/K)$.
For $\ell = 2$, the  Galois groups $U = \mathrm{Gal}(N/L)$ and $V = \mathrm{Gal}(M/L)$ 
are quotients of 
$\rH_0(X_{\mathcal{S}},T) = T_{X_{\mathcal{S}}}$.  For $\ell > 2$, we expect quotients of 
the homology group $\rH_{\ell - 2}(X_{\mathcal{S}},T)$ to appear. 
\end{remark}

We end this introduction with two comments on potential research directions.  First, we remark that though we have restricted ourselves to classical Iwasawa modules,
we expect that the approach we have outlined in this paper will apply to general Selmer groups.  This is already illustrated in the recent work of Lei and Palvannan on Selmer groups of supersingular elliptic curves \cite{LP} and tensor products of Hida families \cite{LP2}.

Secondly, we note that congruences between Eisenstein series and cusp forms play a key role in proofs of one of the divisibilities in main conjectures, whereby the existence of residually-reducible Galois representations with certain ramification behavior leads to lower bounds for the support of Selmer groups.  One can ask how to apply such techniques to directly study the higher codimension behavior of Iwasawa modules. The right hand side of \eqref{eq:simpler1formula} has two terms measuring the size of Galois groups of extensions unramified outside the intersection of two CM types. It would be interesting if one could construct Galois representations that separately control each of the two terms. For instance, one might consider congruences between Hida families modulo Eisenstein ideals attached to $\Lambda$-adic Eisenstein series with constant terms arising from different $p$-adic $L$-functions. 

\section{Duality}

\label{s:duality}

Let $p$ be a prime, let $E$ be a number field, and 
let $F$ be a finite Galois extension of $E$ of prime-to-$p$ degree.  We suppose that $F$ has no real places if $p = 2$.
Let $\Delta = \Gal(F/E)$.  Let $K$ be a Galois
extension of $E$ that is a $\zp^r$-extension of $F$ for some $r \ge 1$, and set $\Gamma = \Gal(K/F)$.  Note that $K/F$ is unramified outside $p$ as a compositum of $\zp$-extensions.  Set $\mc{G} = \Gal(K/E)$ and $\Omega = \zp\ps{\mc{G}}$.

Let $S=S_{p,\infty}$ be the set of all primes of $E$ over $p$ and $\infty$, and let 
$S_f$ be the set of all primes of $E$ over $p$.
For any algebraic extension $F'$ of $F$, let $G_{F',S}$ denote the Galois group of the maximal extension $F'_{S}$ of $F'$ that is unramified outside the primes over $S$.  Let $\mc{Q} = \Gal(F_S/E)$.  
For a compact $\zp\ps{\mc{Q}}$-module $T$, we consider the Iwasawa cochain complex
$$
	\rC_{\Iw}(K,T) = \varprojlim_{F' \subset K} \rC(G_{F',S},T)
$$ 
that is the inverse limit of continuous cochain complexes under corestriction maps, 
with $F'$ running over the finite extensions of $F$ in $K$.  It has the natural structure of a complex of $\Omega$-modules.  We let $\RGa_{\Iw}(K,T)$ denote its class in the derived category and $\rH^i_{\Iw}(K,T)$ its $i$th
cohomology group.  We similarly let
$$
	\rC_{\mf{p},\Iw}(K,T) = \varprojlim_{F' \subset K} \bigoplus_{\mf{P} \mid \mf{p}} \rC(G_{F'_{\mf{P}}},T)
$$
for any $\mf{p} \in S_f$, where $G_{F'_{\mf{P}}}$ denotes the absolute Galois group of the completion $F'_{\mf{P}}$.

For a finitely generated $\Omega$-module, we have 
$$
	\Ext_{\Omega}^i(M,\Omega) \cong \Ext^i_{\zp\ps{\Gamma}}(M,\zp\ps{\Gamma})
$$ 
as $\zp\ps{\Gamma}$-modules
(since $\Omega$ is $\zp\ps{\Gamma}$-projective).  
We employ the notation
$$
	\rE^i(M) = \Ext^i_{\Omega}(M^{\iota},\Omega),
$$ 
where $M^{\iota}$ is the $\Omega$-module $M$ with the new action $\cdot_{\iota}$ given by $f \cdot_{\iota} m = \iota(f)m$ for $f \in \Omega$, where
$\iota \colon \Omega \to \Omega$ is the continuous $\zp$-linear involution given on $\mc{G}$ by inversion.
This is a bit cleaner for the purposes of duality, as it alleviates the need to place involutions in the statements of various results.
We set $M^* = \rE^0(M) = \Hom_{\Omega}(M^{\iota},\Omega)$.

For later use, we note that there are natural isomorphisms of $\Omega$-modules 
\begin{eqnarray*} \label{eq:equalities}
\rE^i(M^\iota) = \rE^i(M)^{\iota} &\mathrm{and}& \rE^i(M(-1)) = \rE^i(M)(1),
\end{eqnarray*}
where $M(n)$ for $n \in \Z$ is the $\Omega$-module that is $M$ with the modified $\mc{G}$-action $g \cdot m = \chi_p^n(g)gm$ for
$\chi_p \colon \mc{G} \to \zp^{\times}$ the $p$-adic cyclotomic character.

Let $\Sigma$ be a subset of $S_f$. Let $\Sigma^c = S_f - \Sigma$.  We let 
$\RGa_{\Sigma,\Iw}(K,T)$ be the class in the derived category of the cone
$$
	\rC_{\Sigma,\Iw}(K,T) = \Cone\left(\rC_{\Iw}(K,T) \to \bigoplus_{v \in \Sigma} \rC_{v,\Iw}(K,T)\right)[-1]
$$
and define $\rH^i_{\Sigma,\Iw}(K,T)$ to be its $i$th cohomology group.
We define $\RGa_{\Sigma^c,\Iw}(K,T)$ and $\rH^i_{\Sigma^c,\Iw}(K,T)$ similarly.

We have the following two spectral sequences.

\begin{proposition} \label{thm:ss}
	Let $T$ be a compact $\zp\ps{\mc{Q}}$-module that is finitely generated and free over $\zp$,
	and let $T^{\#}$ be its $\zp$-dual.
	There are convergent spectral sequences of $\Omega$-modules 
	\begin{eqnarray*}
		&{\rm F}_2^{i,j}(T) = \rE^i(\rH^{3-j}_{\Sigma^c,\Iw}(K,T)) \Rightarrow {\rm F}^{i+j}(T) = 
		\rH^{i+j}_{\Sigma,\Iw}(K,T^\#(1)),&\\
		&{\rm H}_2^{i,j}(T) = \rE^i(\rH^{3-j}_{\Sigma,\Iw}(K,T^\#(1))) \Rightarrow \rH^{i+j}(T) = 
		\rH^{i+j}_{\Sigma^c,\Iw}(K,T).&
	\end{eqnarray*}
\end{proposition}

\begin{proof}
	By definition, we have the commutative diagram of exact triangles (of which we write three terms)
	$$
		\SelectTips{cm}{} \xymatrix{
			\bigoplus_{v \in \Sigma} \RGa_{v,\Iw}(K,T)[-1] \ar@{=}[r] \ar[d] & \bigoplus_{v \in \Sigma} \RGa_{v,\Iw}(K,T)[-1]
			\ar@{-->}[d] \\
			\bigoplus_{v \in S_f} \RGa_{v,\Iw}(K,T)[-1] \ar[r] \ar[d] & \RGa_{S_f,\Iw}(K,T) \ar[r] \ar[d] & \RGa_{\Iw}(K,T) \ar@{=}[d] &\\
			\bigoplus_{v \in \Sigma^c} \RGa_{v,\Iw}(K,T)[-1] \ar[r] & \RGa_{\Sigma^c,\Iw}(K,T) \ar[r]   & \RGa_{\Iw}(K,T),
		}
	$$
	with the dashed arrow being the induced morphism.  The derived Iwasawa-theoretic versions of
	Poitou-Tate and Tate duality found in \cite[Section 8.5]{nekovar} then yield isomorphisms in the derived category of 
	finitely generated $\Omega$-modules
	$$
		\SelectTips{cm}{} \xymatrix{
		\RGa_{\Sigma,\Iw}(K,T) \ar@{-->}[r]^-{\sim} \ar[d] & 
		\RHom_{\Omega}(\RGa_{\Sigma^c,\Iw}(K,T^{\#}(1))^{\iota},\Omega)[-3] \ar[d] \\
		 \RGa_{\Iw}(K,T)  \ar[r]^-{\sim} \ar[d] & \RHom_{\Omega}(\RGa_{S_f,\Iw}(K,T^{\#}(1))^{\iota},\Omega)[-3] \ar[d] \\
		\bigoplus_{v \in \Sigma} \RGa_{v,\Iw}(K,T) \ar[r]^-{\sim} & 
		\bigoplus_{v \in \Sigma} \RHom_{\Omega}(\RGa_{v,\Iw}(K,T^{\#}(1))^{\iota},\Omega)[-2],
		}
	$$
	where the lower two isomorphisms yield the isomorphism of cones.  (That these are morphisms in the derived category 
	of $\Omega$-modules and not simply $\zp\ps{\Gamma}$-modules follows from their definitions 
	and the fact that $\Omega$ is $\zp\ps{\Gamma}$-projective.  The case that $\Delta$ is abelian is treated in \cite{nekovar}, and
	this can be found in a more general context in \cite[Theorem 4.5.1]{LS}.)
\end{proof}

Let us now focus on the case of $\zp(1)$-coefficients.  

\begin{lemma} \label{lem:cohdegs}
	We have $\rH^i_{\Sigma,\Iw}(K,\zp(1)) = 0$ unless $i \in \{1,2,3\}$, and $\rH^3_{\Sigma,\Iw}(K,\zp(1))$ vanishes unless $\Sigma^c$ 
	is empty, in which case it is isomorphic to $\zp$ as an $\Omega$-module.
\end{lemma}

\begin{proof}
	The first statement is a consequence of the fact that $G_{E,S}$ and $G_{E_{\mf{p}}}$ for all $\mf{p} \in S_f$ have $p$-cohomological
	dimension $2$, 
	 the vanishing in degree $0$ following from the fact that $\Gamma$ is infinite.  The first map in the exact sequence
	\begin{equation*}
		\bigoplus_{v \in \Sigma^c} \rH^2_{v,\Iw}(K,\zp(1)) \to \rH^3_{S_f,\Iw}(K,\zp(1)) \to \rH^3_{\Sigma,\Iw}(K,\zp(1)) \to 0
	\end{equation*} 
	is identified via duality (i.e., invariant maps) with the summation map $\bigoplus_{w  \in \Sigma_K^c} \zp \to \zp$, where 
	$\Sigma_K^c$ is the set of places of $K$ over places in $\Sigma^c$.  The second statement follows.
\end{proof}

Let $X_\Sigma$ denote the $\Sigma$-ramified Iwasawa module
over $K$.  Let $X^\flat_\Sigma$ denote the maximal quotient of $X_\Sigma$ that is completely split at the primes in $S_f-\Sigma$.  
We also set 
\begin{equation}
\label{eq:Y}
	Y_\Sigma = \rH^2_{\Sigma,\Iw}(K,\zp(1)).
\end{equation}
For $\mf{p} \in S_f$, let $\mc{G}_{\mf{p}}$ denote the decomposition group in $\mc{G}$ 
 at a place over the prime $\mf{p}$ in $K$, and
set $\mc{K}_{\mf{p}} = \zp\ps{\mc{G}/\mc{G}_{\mf{p}}}$, which has the natural structure of a left $\Omega$-module.
Set
\begin{equation} \label{eq:K}
	\mc{K}_\Sigma = \bigoplus_{\mf{p} \in \Sigma} \mc{K}_{\mf{p}},
\end{equation}
so in particular $\mc{K}_\Sigma = 0$ if $\Sigma = \varnothing$.  Let 
\begin{equation} \label{eq:K0}
	\mc{K}_{\Sigma,0} = \ker(\mc{K}_\Sigma \to \zp)
\end{equation}
be the kernel of the sum of the augmentation maps.

For $\mf{p} \in S_f$, let $\Gamma_{\mf{p}} = \mc{G}_{\mf{p}}\cap \Gamma$ be the decomposition group in $\Gamma$ at a prime over $\mf{p}$ in $K$, and let 
\begin{equation} \label{eq:r}
	r_{\mf{p}} = \rank_{\zp} \Gamma_{\mf{p}}.
\end{equation} 
By \cite[Lem.~4.1.13]{BCGKPST}, we have the following.

\begin{remark} \label{EKp}
	For $j \ge 0$, there are isomorphisms $\rE^j(\mc{K}_{\mf{p}}) \cong (\mc{K}_{\mf{p}}^{\iota})^{\delta_{r_{\mf{p}},j}}$
	of $\Omega$-modules.
\end{remark}

Let $\mf{D}_{\mf{p}}$ denote the Galois group of the maximal abelian, pro-$p$ quotient
of the absolute Galois group of the completion $K_{\mf{p}}$ of $K$ at a prime over $\mf{p}$.  Define $\mf{I}_{\mf{p}}$ to be the inertia
subgroup of $\mf{D}_{\mf{p}}$.  We have completed tensor products 
$$ \label{eq:DpIp}
	D_{\mf{p}} =  \Omega \cotimes{\zp\ps{\mc{G}_{\mf{p}}}} \mf{D}_{\mf{p}}\quad \mr{and}\quad 
I_{\mf{p}} =  \Omega \cotimes{\zp\ps{\mc{G}_{\mf{p}}}} \mf{I}_{\mf{p}}.\
$$
These have the structure of  $\Omega$-modules by left multiplication.  
Set
\begin{eqnarray} \label{eq:DI}
	D_\Sigma = \bigoplus_{\mf{p} \in \Sigma} D_{\mf{p}} &\mr{and}& I_\Sigma = \bigoplus_{\mf{p} \in \Sigma} I_{\mf{p}}.
\end{eqnarray}

\begin{lemma}
\label{lem:needlater}
    There is a canonical exact sequence
   $$
    	0 \to X^\flat_{\Sigma} \to Y_{\Sigma} \to \mc{K}_{\Sigma^c,0} \to 0.
    $$
\end{lemma}
    
\begin{proof}
	We have a long exact sequence
    	\begin{multline*} \label{eq:longer}
    		\bigoplus_{v \in \Sigma^c} \rH^1_{v,\Iw}(K,\zp(1)) \xrightarrow{} \rH^2_{S_f,\Iw}(K,\zp(1)) \to \rH^2_{\Sigma,\Iw}(K,\zp(1)) \\
		\xrightarrow{} \bigoplus_{v \in \Sigma^c} \rH^2_{v,\Iw}(K,\zp(1))
	 	\xrightarrow{} \rH^3_{S_f,\Iw}(K,\zp(1)) \xrightarrow{} \rH^3_{\Sigma,\Iw}(K,\zp(1)).
    	\end{multline*}
    By Poitou-Tate duality, the second term is $X_{S_f}$, and by Tate duality, the first term is $D_{\Sigma^c}$, and the cokernel
    of the resulting restriction map $D_{\Sigma^c} \to X_{S_f}$ is $X^\flat_{\Sigma}$.
    Via the invariant maps of local class field theory, the group $\bigoplus_{v \in \Sigma^c} \rH^2_{v,\Iw}(K,\zp(1))$ is identified with 
    $\mc{K}_{\Sigma^c}$.  Lemma \ref{lem:cohdegs} tells us that $\rH^3_{S_f,\Iw}(K,\zp(1)) = \zp$, and again by class field theory,
    the map $\mc{K}_{\Sigma^c} \to \zp$ is given by summation. 
\end{proof}

In the remainder of this section, we make the following hypothesis:

\begin{hypothesis}
\label{hyp:nicehyp}  The field $K$ contains all $p$-power roots of unity.
\end{hypothesis}

This allows us to pull twists out of our Iwasawa cohomology groups and to apply Weak Leopoldt where helpful.
One could remove this assumption with appropriate modifications, but we do not need to do so for our applications. 

\begin{remark} \label{rem:conseqD=I}
	The canonical surjection $X_{\Sigma} \twoheadrightarrow X^\flat_{\Sigma}$ is an isomorphism if $D_{\Sigma} = I_{\Sigma}$.  
	Our assumption on $K$ implies that $r_{\mf{p}} \ge 1$ for each $\mf{p} \in \Sigma^c$, so the canonical injection 
	$X_{\Sigma} \hookrightarrow Y_{\Sigma}$ has torsion cokernel which is pseudo-null if $r_{\mf{p}} \ge 2$ for each $\mf{p} \in \Sigma^c$.
\end{remark}

Using the spectral sequences of Proposition \ref{thm:ss}, we obtain the following.

\begin{proposition}
\label{prop:iwseq}
	If $\Sigma \notin \{\varnothing, S_f\}$ or $r = 1$, then there is an exact sequence
	\begin{equation} \label{iwseq}
		0 \to \rE^1(Y_{\Sigma^c})(1) \to Y_{\Sigma} \to Y_{\Sigma}^{**} \to 
		\rE^2(Y_{\Sigma^c})(1) \to 0,
	\end{equation}
 	and for $i \ge 1$, there are isomorphisms 
	\begin{equation} \label{iwisoms}
		\rE^i(Y_{\Sigma}^*) \xrightarrow{\sim} \rE^{i+2}(Y_{\Sigma^c})(1)
	\end{equation}
	of $\Omega$-modules.  If $\Sigma = S_f$, then the above statements hold upon localization
	at any prime of $\Omega$ outside the support of $\zp$, while if $\Sigma = \varnothing$, they hold
	outside the support of $\zp(1)$.
	
	More precisely, if $\Sigma = S_f$, then \eqref{iwseq} becomes exact
	upon replacing the rightmost zero by $\zp$, and the maps in \eqref{iwisoms} are isomorphisms for $i \ge 2$.
	For $i = 1$, the map in \eqref{iwisoms} is surjective with procyclic kernel unless it happens that $r = 2$
	and it is injective with finite cyclic cokernel.
\end{proposition}

\begin{proof} 
	Let us first suppose that $\Sigma \notin \{\varnothing,S_f\}$. Consider the spectral sequence 
	$\rF_2^{i,j}(\zp) \Rightarrow \rF^{i+j}(\zp)$ of Proposition \ref{thm:ss}.  By Lemma \ref{lem:cohdegs}
	and the fact that $\Sigma \neq \varnothing$ (resp., $\Sigma \neq S_f$), 
	we have $\rF_2^{i,j}(\zp) = 0$ unless $j \in \{1,2\}$ (resp., $\rF^k(\zp) = 0$ unless $k \in \{1,2\}$).
	The spectral sequence then yields an exact 
	sequence of base terms
	\begin{equation*} \label{ptseq}
		0 \to \rE^1(\rH^2_{\Sigma^c,\Iw}(K,\zp)) \to Y_{\Sigma}
		\to \rH^1_{\Sigma^c,\Iw}(K,\zp)^* \to \rE^2(\rH^2_{\Sigma^c,\Iw}(K,\zp))
		\to 0
	\end{equation*}
	and isomorphisms
	$$
		\rE^i(\rH^1_{\Sigma^c,\Iw}(K,\zp)) \cong \rE^{i+2}(\rH^2_{\Sigma^c,\Iw}(K,\zp))
	$$
	of $\Omega$-modules for $i \ge 1$.  We then obtain our results by applying two isomorphisms: the first 
	\begin{equation*} \label{YSigmadual}
		\rH^1_{\Sigma^c,\Iw}(K,\zp) \cong Y_{\Sigma}^*
	\end{equation*}
	arises from spectral sequence $\rH_2^{i,j}(\zp) \Rightarrow \rH^{i+j}(\zp)$ of Proposition \ref{thm:ss} 
	by the vanishing of the terms $\rH^{i,0}(\zp)$ that occurs since $\Sigma \neq S_f$, and the second
	$$
		\rH^2_{\Sigma^c,\Iw}(K,\zp) \cong Y_{\Sigma^c}(-1)
	$$ 
	follows by our assumption that $K$ contains all $p$-power roots of unity.
	
	If $\Sigma = \varnothing$, then we have 
	$$
		\rF_2^{i,0}(\zp) \cong \rE^i(H^3_{S_f}(K,\zp)) \cong \rE^i(\zp(-1)) \cong \zp(1)^{\delta_{i,r}}
	$$  
	by \cite[Cor.~A.13]{BCGKPST}.
	The above arguments go through so long as we localize all terms at a prime of $\Omega$ outside of the support of $\zp(1)$,
	as well as if $r = 1$.

	For the more precise statements for $\Sigma = S_f$, we can use the results of \cite{BCGKPST}, as we explain.
	Set $U = \rH^1_{\Iw}(K,\zp)$ for brevity of notation.
	As in the proof of \cite[Cor.~4.1.6]{BCGKPST}, we have an exact sequence 
	\begin{equation} \label{ptseqSf}
		0 \to \rE^1(Y_{\varnothing})(1) \to Y_{S_f} \to Y_{S_f}^{**} \to \rE^2(Y_{\varnothing})(1) \to \zp \to 
		\rE^1(U) \to
		\rE^3(Y_{\varnothing})(1) \to 0
	\end{equation}
	and isomorphisms $\rE^i(U) \xrightarrow{\sim} \rE^{i+2}(Y_{\varnothing})$ for $i \ge 2$.  As in \cite[Thm.~4.1.2]{BCGKPST}, 
	we also have an exact sequence
	\begin{equation} \label{problemseq}
		0 \to \zp^{\delta_{r,1}} \to U \to Y_{S_f}^* \to \zp^{\delta_{r,2}}. 
	\end{equation}
	If $r \ge 3$, or if $r = 2$ and the map $Y_{S_f}^* \to \zp$ is zero, 
	we can substitute in the resulting isomorphism $U \cong Y_{S_f}^*$ to give the result.
	If $r \in \{1,2\}$, then the maps $\rE^i(Y_{S_f}^*) \to \rE^{i+2}(Y_{\varnothing})(1)$ are of trivial groups for $i \ge r$ (see
	\cite[Cor.~A.9]{BCGKPST}). 
	For $r = 1$, this implies that the map $\rE^1(U) \to \zp$ given by taking $\Ext$-groups of \eqref{problemseq} is an isomorphism,
	forcing the map $\zp \to \rE^1(U)$ in \eqref{ptseqSf} to also be an isomorphism, hence the result.
	
	Finally, suppose that $r = 2$ and the map $Y_{S_f}^* \to \zp$ of \eqref{problemseq} is nontrivial, hence has image isomorphic to $\zp$.
	Taking $\Ext$-groups, we then have an exact sequence of the form
	$$
		0 \to \rE^1(Y_{S_f}^*) \to \rE^1(U) \to \zp \to 0
	$$
	in which the first term is finite (again by \cite[Cor.~A.9]{BCGKPST}).  Since $\rE^3(Y_{\varnothing})$ is finite as well, 
	it follows that the map $\zp \to \rE^1(U)$
	in \eqref{ptseqSf} must be injective, and so we also have an exact sequence
	$$
		0 \to \zp \to \rE^1(U) \to \rE^3(Y_{\varnothing})(1) \to 0.
	$$
	From these two sequences and a simple application of the snake lemma, we obtain that the composite map $\rE^1(Y_{S_f}^*)
	\to \rE^3(Y_{\varnothing})(1)$ is injective with finite cokernel.
\end{proof}

\begin{corollary} \label{cor:extvanish}
	Suppose that $Y_{\Sigma^c}$ is torsion and $\Sigma \neq S_f$.  Then 
	$\rE^1(Y_{\Sigma})(1) \cong Y_{\Sigma^c}$, and $\rE^i(Y_{\Sigma})(1)$ is zero for all $i \ge 2$.
\end{corollary}

\begin{proof}
	We apply Proposition \ref{prop:iwseq} with $\Sigma$ and $\Sigma^c$ reversed. 
	Note that $\Sigma \neq \varnothing$, since $Y_{S_f} = X_{S_f}$ has nonzero $\zp\ps{\Gamma}$-rank.  
	As $Y_{\Sigma^c}$ is torsion, we have $Y_{\Sigma^c}^* = 0$,
	so $\rE^i(Y_{\Sigma^c}^*) = 0$ for all $i \ge 0$, and the isomorphisms of \eqref{iwisoms} tell us that $\rE^i(Y_{\Sigma}) = 0$ for all
	$i \ge 3$.  Since $Y_{\Sigma^c}^{**} = 0$, the exact sequence \eqref{iwseq} gives the remaining statements.
\end{proof}

\begin{remark} \label{rem:extvanish}
	The result of Corollary \ref{cor:extvanish} remains true for $\Sigma = S_f$ after localization at a prime
	away from the support of $\zp(1)$ (and without localization if $r = 1$), as follows by Proposition \ref{prop:iwseq}.
\end{remark}

Let us set 
\begin{equation} \label{eq:Z}
	Z_{\Sigma} = \begin{cases} \zp & \text{if } \Sigma = S_f \text{ and } r \ge 2, \\ 0 & \text{otherwise}. \end{cases}
\end{equation}  
  
\begin{theorem} \label{thm:free}
	Let $\mf{q}$ be a prime ideal of $\Omega$ outside of the support of $Y_{\Sigma^c}^{\iota}(1) \oplus Z_{\Sigma}$.
	Then the localization $Y_{\Sigma,\mf{q}}$ of $Y_{\Sigma}$ at $\mf{q}$ is a free $\Omega_{\mf{q}}$-module.
\end{theorem}

\begin{proof}
	By Corollary \ref{cor:extvanish} and Remark \ref{rem:extvanish}, we have
	$$
		\Ext^1_{\Omega_{\mf{q}}}(Y_{\Sigma,\mf{q}},\Omega_{\mf{q}}) \cong (\rE^1(Y_{\Sigma})^{\iota})_{\mf{q}} \cong
		(Y_{\Sigma^c}(-1)^{\iota})_{\mf{q}} \cong (Y_{\Sigma^c}^{\iota}(1))_{\mf{q}} = 0,
	$$
	and
	$\Ext^i_{\Omega_{\mf{q}}}(Y_{\Sigma,\mf{q}},\Omega_{\mf{q}}) = 0$ for all $i \ge 2$.  
	Since $Y_{\Sigma,\mf{q}}$ is a finitely generated module over the regular local ring $\Omega_{\mf{q}}$ with vanishing 
	higher Ext-groups to $\Omega_{\mf{q}}$, it is free (cf. \cite[(4.12)]{AB}).
\end{proof}

\begin{proposition} \label{prop:localseq} 
	For any nonempty subset $P$ of $\Sigma$, we have a map of exact sequences 
	$$
		\SelectTips{cm}{} \xymatrix{
		0 \ar[r] & \rE^1(\mc{K}_P)(1) \ar[r] \ar[d] &
		D_P  \ar[r] \ar[d] & D_P^{**} \ar[r] \ar[d] & 
		\rE^2(\mc{K}_P)(1) \ar[d] \ar[r] & 0 &\\
		0 \ar[r] & \rE^1(Y_{\Sigma^c})(1) \ar[r] & Y_{\Sigma} \ar[r] & Y_{\Sigma}^{**} \ar[r] & 
		 \rE^2(Y_{\Sigma^c})(1) \ar[r] & Z_{\Sigma}
		}
	$$
	of $\Omega$-modules in which the vertical maps are the canonical ones.  If the primes of $K$ over each $\mf{p} \in P$ 
	have infinite residue field degree, then $D_P = I_P$ and $\rE^1(\mc{K}_P) = 0$.
\end{proposition}

\begin{proof} The exactness of the lower sequence was shown in Proposition
\ref{prop:iwseq}.  The exactness of the upper sequence is shown in \cite[Thm.~4.1.14]{BCGKPST} via the spectral sequence of 
derived Tate duality (see \eqref{Tate} below), and the map of exact sequences from the corresponding map of spectral sequences.  
That $D_P = I_P$ 
is \cite[Lem.~4.2.2]{BCGKPST}, and $\rE^1(\mc{K}_P) = 0$ follows from Remark \ref{EKp} and $r_{\mf{p}} \ge 2$ (since $K$
is assumed to contain all $p$-power roots of unity and its completion at $\mf{p}$ to contain the unramified $\zp$-extension).
\end{proof}

Let us refine the above result in the local setting.

\begin{lemma} \label{lem:localexts}
	Let $\mf{p} \in S_f$.  
	The $\Omega$-module $D_{\mf{p}}$ has rank $d_{\mf{p}} = [E_{\mf{p}}:\mathbb{Q}_p]$.
	We have $\rE^i(D_{\mf{p}}) = 0$ unless $i \in \{0,1,r_{\mf{p}}-2\}$.  Moreover, the following statements hold.
	\begin{enumerate}
		\item[(i)] If $r_{\mf{p}} = 1$, then $D_{\mf{p}}^* \cong D_{\mf{p}}^{**}(-1)$ is
		$\Omega$-free 
		and fits in an exact sequence
		$$
			0 \to \mc{K}_{\mf{p}}^{\iota} \to D_{\mf{p}}(-1) \to D_{\mf{p}}^* \to 0,
		$$
		and $\rE^1(D_{\mf{p}}) \cong \mc{K}_{\mf{p}}(-1)$.
		\item[(ii)] If $r_{\mf{p}} = 2$, then $D_{\mf{p}}^* \cong D_{\mf{p}}^{**}(-1)$ 
		is $\Omega$-free and fits in an exact sequence
		$$
			0 \to D_{\mf{p}}(-1) \to D_{\mf{p}}^* \to \mc{K}_{\mf{p}}^{\iota} \to 0,
		$$
		and $\rE^1(D_{\mf{p}}) \cong \mc{K}_{\mf{p}}(-1)$.
		\item[(iii)] If $r_{\mf{p}} = 3$, then $D_{\mf{p}}^* \cong D_{\mf{p}}(-1)$, and
		there is an exact sequence
		$$
			0 \to \mc{K}_{\mf{p}}(-1) \to \rE^1(D_{\mf{p}}) \to \mc{K}_{\mf{p}}^{\iota} \to 0.
		$$
		\item[(iv)] If $r_{\mf{p}} \ge 4$, then $D_{\mf{p}}^* \cong D_{\mf{p}}(-1)$, 
		$\rE^1(D_{\mf{p}}) \cong \mc{K}_{\mf{p}}(-1)$, and 
		$\rE^{r_{\mf{p}}-2}(D_{\mf{p}}) \cong \mc{K}_{\mf{p}}^{\iota}$.
	\end{enumerate}
\end{lemma}

\begin{proof} 
	The local spectral sequence in the proof of Proposition \ref{thm:ss} for $T = \zp$ has the form
	\begin{equation} \label{Tate}
		\rE^i(\rH_{\mf{p},\Iw}^{2-j}(K,\zp(1))) \Rightarrow  \rH_{\mf{p},\Iw}^{i+j}(K,\zp).
	\end{equation}
	We have $\rH_{\mf{p},\Iw}^i(K,\zp) \cong \rH_{\mf{p},\Iw}^i(K,\zp(1))(-1)$ 
	by assumption on $K$,
	and $\rH^i_{\mf{p},\Iw}(K,\zp(1))$ is trivial unless $i \in \{1,2\}$.
	Since
	\begin{eqnarray*} \label{eq:fits}
		\rH^1_{\mf{p},\Iw}(K,\zp(1)) \cong D_{\mf{p}} &\mr{and}& \rH^2_{\mf{p},\Iw}(K,\zp(1)) \cong \mc{K}_{\mf{p}},
	\end{eqnarray*}
	the spectral sequence \eqref{Tate} yields an exact sequence
	\begin{equation} \label{locales}
		0 \to \rE^1(\mc{K}_{\mf{p}}) \to D_{\mf{p}}(-1) \to D_{\mf{p}}^* \to \rE^2(\mc{K}_{\mf{p}}) 
		\to \mc{K}_{\mf{p}}(-1) \to \rE^1(D_{\mf{p}}) \to \rE^3(\mc{K}_{\mf{p}}) \to 0
	\end{equation}
	and isomorphisms $\rE^i(D_{\mf{p}}) \xrightarrow{\sim} \rE^{i+2}(\mc{K}_{\mf{p}})$ for $i \ge 2$.
	The exact sequences and isomorphisms follow easily from this and Remark \ref{EKp}.  
	(Here, one must note
	that the map $\mc{K}_{\mf{p}}^{\iota}(1) \to \mc{K}_{\mf{p}}$ that arises in \eqref{locales}
	for $r_{\mf{p}} = 2$ can only be zero,
	as in the proof of \cite[Thm.~4.1.14]{BCGKPST} already cited.)
	The statements of freeness for $r_{\mf{p}} \in \{1,2\}$ follow from $\rE^i(D_{\mf{p}}^*) = 0$ for $i \ge 1$,
	which is derived from the above and \cite[Cor.~A.9]{BCGKPST}.
	The equality $\rank_{\Omega} D_{\mf{p}} = d_{\mf{p}}$ follows from \cite[Lem.~4.3.1(b)]{BCGKPST}.  
\end{proof}
\medbreak

We note that Lemma \ref{lem:localexts} tells us that the reflexive $\Omega$-module $D_{\mf{p}}^*$ is not free if $r_{\mf{p}} \ge 3$, since in that case its first Ext-group is nonzero.
The following corollary is proven in the same manner as Theorem \ref{thm:free} but using Lemma \ref{lem:localexts}.

\begin{corollary} \label{cor:localfree}
	Let $\mf{p} \in S_f$, and let $\mf{q}$ be a prime ideal of $\Omega$ that is either 
	\begin{itemize}
		\item of codimension less than $r_{\mf{p}}$ or
		\item outside the support of $\mc{K}_{\mf{p}}^{\iota}(1)$ and, if $r_{\mf{p}} \ge 3$, also
		outside the support of $\mc{K}_{\mf{p}}$.
	\end{itemize}
	Then $(D_{\mf{p}})_{\mf{q}}$ is free of rank $[E_{\mf{p}}:\mathbb{Q}_p]$ over $\Omega_{\mf{q}}$.
\end{corollary}

\section{CM fields}
\label{s:CM}

Unless otherwise stated, we maintain the notation of the previous section.
Let $E$ be a CM extension of $\Q$ of degree $2d$ and $E^+$ its maximal totally real subfield.  
Let $p$ be an odd prime such that each prime over $p$ in $E^+$ splits in $E$. 
By a ($p$-adic) CM type, we shall mean a set consisting of one prime of $E$ over each of the primes over $p$ in $E^+$.

Let $\widetilde{E}$ be the compositum of all $\zp$-extensions of $E$. If Leopoldt's conjecture holds for $E$ and $p$, 
then  $\widetilde{E}$ is the compositum of the cyclotomic $\zp$-extension $E^{\mr{cyc}}$ and the anticyclotomic
 $\zp^d$-extension $E^{\mr{acyc}}$ of $E$. We set $\Gamma = \Gal(\widetilde{E}/E)$.
 
As before, we let $r = \rank_{\zp} \Gamma$ and $r_{\mf{p}} = \rank_{\zp} \Gamma_{\mf{p}}$, and we also set 
$d_{\mf{p}} = [E_{\mf{p}}:\qp]$ for $\mf{p} \in S_f$.
 
\begin{lemma} \label{lem:sweepup} Let $\mf{p} \in S_f$.
\begin{enumerate}
	\item[(i)] One has $r_{\mf{p}} = d_{\mf{p}}+1$.
	\item[(ii)] The extension $\widetilde{E}/E$ has infinite residue field degree at $\mf{p}$.
\end{enumerate}
\end{lemma}

\begin{proof}
Let $\Sigma$ be a CM type containing $\mf{p}$.  To prove (ii), it suffices to show that $\mf{p}$ has infinite order in the inverse limit of the ray class groups of $E$ of conductor a power of $\prod_{\mf{q} \in \overline{\Sigma}} \mf{q}$.  Let $\alpha \in \mc{O}_E$ generate a positive power of $\mf{p}$.  By  class field theory, it suffices to prove 
that no positive power of $\alpha $ lies in the closure $U$ of the image of the unit group $\mc{O}_E^{\times}$ in
$\prod_{\mf{q} \in \overline{\Sigma}} E_{\mf{q}}^{\times}$.  Here $\prod_{\mf{q} \in \overline{\Sigma}} E_{\mf{q}}^{\times}$ is canonically isomorphic to $(E^+ \otimes_{\Z} \zp)^{\times}$, so the norm $\mathrm{Norm}_{E^+/\Q}$
from $E^+$ to $\Q$ induces a continuous homomorphism $\mc{N} \colon \prod_{\mf{q} \in \overline{\Sigma}} E_{\mf{q}}^{\times} \to \qp^{\times}$. The group $\mc{O}_{E^+}^{\times}$ has finite index in $\mc{O}_E^{\times}$
and $\mathrm{Norm}_{E^+/\Q} (\mc{O}_{E^+}^{\times}) \subseteq  \{\pm 1\}$, so $U \cap \ker \mc{N}$ is of finite index in $U$.  Let $T$ be the set of embeddings of $E$ into $\overline{\qp}$ that send some prime in $\overline{\Sigma}$ into the maximal ideal of the integral closure of
$\zp$ in $\overline{\qp}$.   Then $\mc{N}(\alpha) = \prod_{\sigma \in T} \sigma(\alpha)$ is a product of non-units of the ring of all algebraic integers, so is certainly not a root of unity.  Thus, no positive power of $\alpha$ lies in $\ker \mc{N}$, so no such power lies  in $U$ and we have (ii). 

From (ii), we see that $r_{\mf{p}} = \rank_{\zp} J_{\mf{p}} + 1$, where $J_{\mf{p}}$ denotes the inertia group in $\Gamma_{\mf{p}}$. 
	Local reciprocity maps provide a homomorphism $\bigoplus_{\mf{q} \in S_f} \mc{O}_{E_{\mf{q}}}^{\times}   \to \Gamma$ with
	kernel $\mc{O}_E^{\times} \otimes_{\Z} \zp$ and finite cokernel. In particular,
	$\rank_{\zp} J_{\mf{q}} \le  \rank_{\zp} \mc{O}_{E_{\mf{q}}}^{\times} = d_{\mf{q}}$ for all $\mf{q} \in S_f$.
	As the $(-1)$-eigenspace of $\mc{O}_E^{\times} \otimes_{\Z} \zp$ under complex conjugation is finite,
	the sum of the $\zp$-ranks of the inertia subgroups at $\mf{q} \in \Sigma$ in $\Gal(E^{\mr{acyc}}/E)$ is $d$.
	As $\sum_{\mf{q} \in \Sigma} d_{\mf{q}} = d$, this forces $\rank J_{\mf{q}} = d_{\mf{q}}$ for all $\mf{q} \in \Sigma$. 
	In particular, we have (i).
\end{proof}

We let $\psi$ denote a one-dimensional character of the absolute Galois group of $E$ of finite order prime to $p$, and
we let $E_{\psi}$ denote the fixed field of its kernel.
We set $F = E_{\psi}(\mu_p)$ and $\Delta = \Gal(F/E)$.
Let $\omega$ denote the Teichm\"uller character of $\Delta$. 
We set $K = F\widetilde{E}$.  We take $\mc{G} = \Gal(K/E)$.  We shall make the identification $\Gamma = \Gal(K/F)$ for
the isomorphism given by restriction.

Let $W$ denote the Witt vectors of $\overline{\mathbb{F}}_p$.  We set 
$\Omega = \zp\ps{\mc{G}}$ and $\Lambda = W\ps{\Gamma}$.   
For a compact $\Omega$-module $A$, we define
\begin{equation} \label{eq:eig}
	A^{\psi} = A \cotimes{\zp[\Delta]} W
\end{equation} 
for the map $\zp[\Delta] \to W$
induced by $\psi$.  In particular, we have $\Omega^{\psi} \cong \Lambda$.
When dealing with finitely generated $\Lambda$-modules $M$, we abuse notation and set $\rE^j(M) = \Ext_{\La}^j(M^{\iota},\La)$, 
much as before but now with $W$-coefficients.

For any subset $P$ of $S_f$, let us set 
\begin{equation} \label{eq:d}
	d_P = \sum_{\mf{p} \in P} d_{\mf{p}}.
\end{equation}

\begin{lemma} \label{lem:ranklemma}
	Let $\mc{S}$ be a subset of $S_f$ containing a  CM type $\Sigma$, let $\mc{S}^c = S_f-\mc{S}$, 
	and let $\mc{T} = \mc{S}-\Sigma$.  We have
	$$
		\rank_{\Lambda} X_{\mc{S}}^{\psi} = \rank_{\Lambda} Y_{\mc{S}}^{\psi} = d - d_{\mc{S}^c},
	$$
	where $Y_{\mc{S}}$ is as in (\ref{eq:Y}).
	Moreover, the canonical map $I_{\mc{T}}^{\psi} \to X_{\mc{S}}^{\psi}$ is injective with torsion cokernel.
\end{lemma}

\begin{proof} We first note that $X_{\mc{S}} = X^\flat_{\mc{S}}$ because of Lemma \ref{lem:sweepup}(ii).  
By Lemma \ref{lem:needlater}, the cokernel of the injection
	$X_{\mc{S}} \hookrightarrow Y_{\mc{S}}$ is isomorphic to the $\Lambda$-torsion module $\mc{K}_{\mc{S}^c,0}$ (noting 
	$\Gamma_{\mf{p}} \neq 0$).  Therefore the ranks of $X_{\mc{S}}^{\psi}$ and $Y_{\mc{S}}^{\psi}$ are the same.

	We know that $X_{S_f}^{\psi}$ has $\Lambda$-rank $d = r_2(E)$ by \cite[Lem.~4.3.1(a)]{BCGKPST}, and $X_{\Sigma}^{\psi}$
	is $\Lambda$-torsion by the work of Hida-Tilouine \cite[Thm.~1.2.2]{HT}.
	For any subset $P$ of $S_f$, we have
	\begin{equation} \label{eq:rankID}
		\rank_{\Lambda} I_P^{\psi} =  \rank_{\Lambda} D_P^{\psi} = d_P
	\end{equation}
	by Lemma \ref{lem:sweepup}, Proposition \ref{prop:localseq}, and \cite[Lem.~4.3.1(b)]{BCGKPST}.  Since
	$d_{\Sigma^c} = d$, this forces $I_{\Sigma^c}^{\psi}$ to have image of rank $d$ in $X_{S_f}^{\psi}$. 
	As $\mc{S}^c \subset \Sigma^c$, the image of $I_{\mc{S}^c}^{\psi}$ in $X_{S_f}^{\psi}$ has rank $d_{\mc{S}^c}$, and therefore
	$X_{\mc{S}}^{\psi} = \coker(I_{\mc{S}^c}^{\psi} \to X_{S_f}^{\psi})$ has rank $d-d_{\mc{S}^c}$.
	
	Similarly, since $X_{\Sigma}^{\psi}$ is $\Lambda$-torsion, 
	the image of $I_{\mc{T}}^{\psi}$ in $X_{\mc{S}}^{\psi}$ must have $\Lambda$-rank $d_{\mc{T}} = d - d_{\mc{S}^c}$, 
	and the kernel of the map $I_{\mc{T}}^{\psi} \to X_{\mc{S}}^{\psi}$ is then $\Lambda$-torsion.  
	On the other hand, the $\Lambda$-torsion in $I_{\mc{T}}^{\psi}$ is isomorphic to a subgroup of $(\rE^1(\mc{K}_{\mc{T}})(1))^{\psi}$
	by Proposition \ref{prop:localseq}, but the latter group is zero by Remark \ref{EKp} since $r_{\mf{p}} \ge 2$ for all $\mf{p} \in S_f$ by 
	Lemma \ref{lem:sweepup}.
\end{proof}

As mentioned, for a  CM type $\Sigma$, the $\Lambda$-module $X_{\Sigma}^{\psi}$ is torsion.  We will use 
$\mc{L}_{\Sigma,\psi}$ to denote a generator of $c_1(X_{\Sigma}^{\psi})$.  
The Iwasawa main conjecture for $\Sigma$ and the character $\psi$ states that $\mc{L}_{\Sigma,\psi}$ can be taken to be the
Katz $p$-adic $L$-function for $\Sigma$ and $\psi$ (or more precisely a power series that determines it).  

For $\mf{p} \in S_f$, let $\Delta_{\mf{p}}$ be the decomposition group in $\Delta = \mr{Gal}(F/E)$.
We have $\mc{K}_{\mf{p}}^{\psi} = 0$ unless $\psi|_{\Delta_{\mf{p}}} = 1$, in which case 
$\mc{K}_{\mf{p}}^{\psi} \cong  W\ps{\Gamma/\Gamma_{\mf{p}}}$.  It follows from Remark \ref{EKp} that
$$
	\rE^j(\mc{K}_{\mf{p}})(1)^{\psi} \cong \rE^j(\mc{K}_{\mf{p}}^{\omega\psi^{-1}})(1) 
	\cong ((\mc{K}_{\mf{p}}^{\omega\psi^{-1}})^{\iota}(1))^{\delta_{r_{\mf{p}},j}}
$$
is zero unless $j = r_{\mf{p}}$ and $\omega\psi^{-1}|_{\Delta_{\mf{p}}} = 1$.  If nonzero, the latter $\Lambda$-module is isomorphic
to $W\ps{\Gamma/\Gamma_{\mf{p}}}^{\iota}(1)$. 

\begin{remark} \label{rem:noiota}
	In fact, $W\ps{\Gamma/\Gamma_{\mf{p}}}$ and $W\ps{\Gamma/\Gamma_{\mf{p}}}^{\iota}$
	are isomorphic as $\Lambda$-modules via the continuous $\zp$-linear map that takes a group element to its inverse.  
	In particular, we have $\mc{K}_{\mf{p}}^{\psi} \cong (\mc{K}_{\mf{p}}^{\psi})^{\iota}$ as $\Lambda$-modules,
	while $\mc{K}_{\mf{p}}^{\psi^{-1}} \cong (\mc{K}_{\mf{p}}^{\psi})^{\iota}$ as $\Lambda[\Delta]$-modules.
\end{remark}

\begin{remark} \label{rem:explicit}
	Choose a topological generating set $\{\gamma_1, \ldots, \gamma_r\}$ of $\Gamma$ so that for $s = r_{\mf{p}}$, we have
	$\Gamma_{\mf{p}} = \langle \gamma_1^{q_1}, \ldots, \gamma_s^{q_s} \rangle$ for some $p$-powers $q_1, \ldots, q_s$.
	Identifying $W\ps{\Gamma}$ with $W\ps{T_1,\ldots,T_r}$ via the continuous $W$-linear isomorphism taking $\gamma_i-1$
	to $T_i$ for $1 \le i \le r$,
	we then have
	\begin{equation*} \label{eq:fv}
		W\ps{\Gamma/\Gamma_{\mf{p}}} \cong W\ps{T_1,\ldots,T_r}/((T_1+1)^{q_1}-1, \ldots, (T_s+1)^{q_s}-1).
	\end{equation*}
	The codimension $s$ primes of $W\ps{\Gamma}$ in the support of the latter module have the form
	$$
		(\Phi_{q'_1}(T_1+1), \ldots, \Phi_{q'_s}(T_s+1)),
	$$
	where $q'_i$ is a positive divisor of $q_i$ for each $i$, and $\Phi_n$ is the $n$th cyclotomic polynomial.
	
	As for $W\ps{\Gamma/\Gamma_{\mf{p}}}^{\iota}(1)$, under this identification, we have
	$$
		W\ps{\Gamma/\Gamma_{\mf{p}}}^{\iota}(1) \cong W\ps{T_1, \ldots, T_r}/(\chi_p(\gamma_1)^{-q_1}(T_1+1)^{q_1}-1,
		\ldots, \chi_p(\gamma_s)^{-q_s}(T_s+1)^{q_s}-1),
	$$F
	where $\chi_p$ denotes the $p$-adic cyclotomic character on $\Gamma$.
\end{remark}

\begin{remark} \label{trivzero}
	For a CM type $\Sigma$, the primes in the support of $\mc{K}_{\mf{p}}^{\psi}$ for $\mf{p} \in \overline{\Sigma}$ and the primes 
	in the support of $(\mc{K}_{\mf{p}}^{\omega\psi^{-1}})^{\iota}(1)$ for $\mf{p} \in \Sigma$ yield trivial zeros of the Katz 
	$p$-adic $L$-functions for $\Sigma$ and $\psi$ (cf. \cite[Sect.~5.3]{Katz}).  In our terminology, this says that 
	$\mc{L}_{\Sigma,\psi}$ lies in each of these primes.
\end{remark}
	
\section{Exterior powers} \label{sec:exterior}

In this section, we prove some abstract lemmas on exterior powers that we shall use in our study.  We fix an integral domain $R$.
For a finitely generated $R$-module $M$, let $\bigwedge^{\ell} M$ denote the $\ell$th exterior power of $M$
over $R$.  Let $\rT_n(M)$ denote the maximal submodule of $M$ that is supported in codimension at least $n$.  Let $\Fitt(M)$ denote the $0$th Fitting ideal of $M$. For brevity of notation, we set $\rQ(M) = R/\Fitt(M)$ and $M_{\tf} = M/\rT_1(M)$.  We use the notation $c_n(M)$ for the $n$th Chern
class if the support of $M$ has codimension at least $n$ and set $t_n(M) = c_n(\rT_n(M))$ in general.  We will identify $t_1(M)$ with the
usual characteristic ideal of the torsion submodule $T_1(M)$ of $M$.

Let $\mc{X}$ and $\mc{F}$ be $R$-modules of rank $\ell \ge 1$ with $\mc{F}$ free. Let
$\lambda \colon \mc{X}\to \mc{F}$ be an $R$-module homomorphism with torsion kernel $\rT_1(\mc{X})$ and torsion cokernel
$\mc{E}$, which in our applications will be pseudo-null. 
The induced homomorphism
$\bigwedge^\ell\lambda \colon \bigwedge^\ell \mc{X} \to \bigwedge^\ell \mc{F}$ on exterior powers
fits in an exact sequence
$$
\textstyle	0 \to \rT_1(\bigwedge^\ell \mc{X}) \to \bigwedge^{\ell} \mc{X} \xrightarrow{\bigwedge^{\ell} \lambda} \bigwedge^{\ell} \mc{F}
	\to \rQ(\mc{E}) \to 0,
$$
essentially by definition. We note that if $\mc{I}$ is an $R$-submodule of $\mc{X}$ of rank $\ell$, then the induced map 
$(\bigwedge^{\ell} \mc{I})_{\tf} \to (\bigwedge^{\ell} \mc{X})_{\tf}$ on maximal torsion-free quotients is injective, so we can and do 
identify $(\bigwedge^{\ell} \mc{I})_{\tf}$ with its image in $(\bigwedge^{\ell} \mc{X})_{\tf}$.
 
\begin{lemma}
\label{lem:torsion!}
Suppose that $R$ is a Noetherian UFD.
For $n \ge 1$ and $1 \le i \le n$, let $\mc{I}_i$ be a rank $\ell$ submodule of $\mc{X}$ mapped injectively under $\lambda$ 
into a free submodule $\mc{J}_i$ of $\mc{F}$ with pseudo-null cokernel $\mc{B}_i := \mc{J}_i/\lambda(\mc{I}_i)$. 
Let 
$\theta_0$, $\theta_1$, and $L_i$
be generators of
of $t_1(\mc{X})$, $c_1(\mc{E})$, and $c_1(\mc{X}/\mc{I}_i)$,
respectively.   
Then $\theta_0$ divides $L_i$, and $\tilde{L}_i = 
\theta_1 L_i/\theta_0$ generates $c_1(\mc{F}/\mc{J}_i) = c_1(\bigwedge^\ell \mc{F}/\bigwedge^\ell \mc{J}_i)$.

We have an exact sequence
\begin{equation*} \label{eq:inertia}   			
    		\frac{(\bigwedge^\ell  \mc{X})_{\tf}}{(\bigwedge^\ell  \mc{I}_1)_{\tf} + \cdots + (\bigwedge^\ell  \mc{I}_n)_{\tf} } \to
		\frac{ \bigwedge^\ell  \mc{F}}{\bigwedge^\ell  \mc{J}_1 + \cdots + \bigwedge^\ell  \mc{J}_n } \to 
		\frac{\rQ(\mc{E})}{(\tilde{L}_1,\ldots,\tilde{L}_n)\rQ(\mc{E})}\to 0, \\   		
\end{equation*}
where the leftmost map has pseudo-null kernel with support contained in that of the $\Lambda$-modules $\rQ(\mc{B}_i)$.
\end{lemma}

\begin{proof}  
The existence of and statements about $\theta_0$, $\theta_1$, and $L_i$ follow from the assumption that $R$ is a UFD.
For $1 \le i \le n$, since $\mc{I}_i \to \mc{J}_i$ is injective with pseudo-null cokernel, the sequence of morphisms 
$$0 \to \rT_1(\mc{X}) \to \frac{\mc{X}}{\mc{I}_i} \to \frac{\mc{F}}{\mc{J}_i} \to \mc{E} \to 0$$
is exact when localized at any codimension one prime of $R$.  
We conclude that 
\begin{equation}
\label{eq:Liformula}
c_1(\mc{F}/\mc{J}_i) = c_1(\mc{X}/\mc{I}_i) + c_1(\mc{E}) - t_1(\mc{X}) = 
c_1(R/R\tilde{L}_i).
\end{equation}
Since $\mc{J}_i$ and $\mc{F}$ are free of rank $\ell$, we see from
 \eqref{eq:Liformula} that the exterior power $\bigwedge^\ell  \mc{J}_i$ is equal to the free rank one submodule 
 $\tilde{L}_i \cdot \bigwedge^\ell \mc{F}$ of $\bigwedge^\ell \mc{F}$.

We have a commutative diagram of $R$-modules with exact rows
 \begin{equation}
    \label{eq:diagram1newaaa}
    		\SelectTips{cm}{} \xymatrix{
    		&
    		\bigoplus_{i=1}^n \bigwedge^\ell I_i \ar[r] \ar[d] & \bigoplus_{i=1}^n \bigwedge^\ell \mc{J}_i 
		\ar[r]^{h} \ar[d]_{g}  &\bigoplus_{i=1}^n \rQ(\mc{B}_i) \ar[r]\ar[d]_{g'}&0\\
    		 0 \ar[r]& (\bigwedge^\ell \mc{X})_{\tf} \ar[r] & \bigwedge^\ell \mc{F} \ar[r] & 
		\rQ(\mc{E}) \ar[r]&0.\\ 
}
    \end{equation}
We can pick generators for the free rank one $R$-modules $\bigwedge^\ell \mc{J}_i$ and $\bigwedge^\ell \mc{F}$
so that the map $g \colon R^n \to R$ has the form $g(\alpha_1,\ldots,\alpha_n) =\sum_{i=1}^n \tilde{L}_i \alpha_i$. 
The snake lemma then yields an exact sequence of $R$-modules on cokernels as in the statement, where the kernel of the first map is the cokernel of the map $\ker g \to \ker g'$ induced by $h$.
\end{proof}

\begin{corollary}
\label{cor:pseudo}
In the notation of Lemma \ref{lem:torsion!}, there are isomorphisms
$$
	\mc{N} := \frac{ \bigwedge^\ell  \mc{F}}{\bigwedge^\ell \mc{J}_1 + \cdots + \bigwedge^\ell  \mc{J}_n } \cong \frac{R}{(\tilde{L}_1, \ldots,  \tilde{L}_n)} \cong \frac{R\theta_0 }{R \theta_1 L_1  + \cdots + R\theta_1 L_n }.$$
Let $\theta$ be a gcd in $R$ of $L_1, \ldots, L_n$.  Then $\theta_0$ divides $\theta$, so  $\nu = \theta_1  \theta/\theta_0$ is in $R$.  The maximal pseudo-null submodule of  $\mc{N}$ is 
$$ \rT_2(\mc{N}) = \nu \mc{N} \cong \frac{R \theta_1 \theta  }{R \theta_1 L_1 + \cdots + R \theta_1 L_n} \cong \frac{R}{ (L_1/\theta, \ldots, L_n/\theta)},$$
and we have an exact sequence of pseudo-null modules 
 \begin{equation}
    \label{eq:inertiatwo}   			
    		0\to \frac{\ker g'}{h(\ker g)} \to
		\rT_2\left ( \frac{(\bigwedge^\ell  \mc{X})_{\tf}}{(\bigwedge^\ell  \mc{I}_1)_{\tf} + \cdots + (\bigwedge^\ell  \mc{I}_n)_{\tf} } \right ) \to
		\nu \mc{N} \to \frac{\nu \rQ(\mc{E})}{(\tilde{L}_1, \ldots, \tilde{L}_n) \rQ(\mc{E})} \to 0 \\   		
\end{equation}
where $g$, $g'$, and $h$ are as in \eqref{eq:diagram1newaaa}.
In particular, if $\rQ(\mc{B}_i)=0$ for all $i$ then $\tilde{L}_i\in \Fitt(\mc{E})$ for all $i$ and (\ref{eq:inertiatwo}) becomes a short exact sequence 
\begin{equation}
    \label{eq:inertiathree}   			
    		0\to \rT_2\left ( \frac{(\bigwedge^\ell  \mc{X})_{\tf}}{(\bigwedge^\ell  \mc{I}_1)_{\tf} + \cdots + (\bigwedge^\ell  \mc{I}_n)_{\tf} } \right ) \to
		\nu \mc{N} \to \nu \rQ(\mc{E}) \to 0. \\   		
\end{equation}
\end{corollary}

\begin{remark} \label{rem:pseudo}
	From the proof of Lemma \ref{lem:torsion!}, and in particular diagram \eqref{eq:diagram1newaaa}, we see that
	$\tilde{L}_i \cdot \Fitt(\mc{B}_i) \subset \Fitt(\mc{E})$, and the kernel of the first term of the exact sequence in 
	\eqref{eq:inertiatwo} is the cokernel of the map
	\begin{equation*} \label{eq:nicecoker}
		\ker(R^n \xrightarrow{g} R) \xrightarrow{h} \ker\left(\bigoplus_{i=1}^n \rQ(\mc{B}_i) \xrightarrow{g'} \rQ(\mc{E})\right)
	\end{equation*}
	where $g(\alpha_1, \ldots, \alpha_n) = \sum_{i=1}^n \tilde{L}_i \alpha_i$, the map $h$ is induced by the canonical quotient
	map $R^n \to \bigoplus_{i=1}^n \rQ(\mc{B}_i)$, and $g'$ is the map induced by $g$. 
	Alternatively, we have
	\begin{multline*}
	\ker(g') = \left\{(\overline{\alpha}_i)_i \;\Big|\; \sum_{i=1}^n \tilde{L}_i \alpha_i \in \Fitt(\mc{E})\right\} \\ \supset h(\ker(g)) = 
	\left\{(\overline{\alpha}_i)_i \;\Big|\; \sum_{i=1}^n \tilde{L}_i \alpha_i \in \sum_{i=1}^n \tilde{L}_i \Fitt(\mc{B}_i) \right\},
	\end{multline*}
	where $\overline{\alpha}_i$ denotes the image of $\alpha_i \in R$ in $\rQ(\mc{B}_i)$.
\end{remark}

\section{Main theorems}
\label{s:maintheorems}

We keep the notation and assumptions of Section \ref{s:CM}. 
That is, we work with a CM field $E$ of degree $2d$, a prime $p$ such that
all primes over it split in $E/E^+$, and a $p$-adic character $\psi$ of the absolute Galois group of $E$. We again have
\begin{itemize}
	\item the fields $F = E_{\psi}(\mu_p)$ and $K = F\widetilde{E}$ for the compositum $\widetilde{E}$ of $\zp$-extensions of $E$,
	\item the Galois groups $\mc{G} = \Gal(K/E)$ and $\Gamma = \Gal(\widetilde{E}/E)$, and 
	\item the Iwasawa algebras $\Omega = \zp\ps{G}$ and $\Lambda = W\ps{\Gamma}$ for $W$ the Witt vectors of $\overline{\F_p}$. 
\end{itemize}
For the definitions of the Iwasawa modules $X_P$, $X^{\flat}_P$, $Y_P$, $\mc{K}_P$, $\mc{K}_{P,0}$, $I_P$, and $D_P$, ranks $r_P$, and degrees $d_P$ attached to subsets $P$ of the set $S_f$ of primes over $p$, we refer the reader to \eqref{eq:Y}-\eqref{eq:DI} and just prior, as well as to \eqref{eq:d}.
Recall that for a compact $\Omega$-module $A$, we denote by $\bigwedge^\ell A^\psi$ the $\ell$th exterior power over $\Lambda$ of 
the eigenspace $A^\psi$ of
$A$ defined in \eqref{eq:eig}.  Moreover, if $A^\psi$ is a finitely generated $\Lambda$-module, then $\Fitt(A^\psi)$ denotes its $0$th Fitting ideal in $\Lambda$.

For $n \ge 1$, let $\mc{S}_1, \ldots, \mc{S}_n$ be distinct CM types of primes over $p$
viewed as subsets of the set $S_f$ of all primes over $p$ in $E$.  Let 
$$
	\mc{S} = \bigcup_{i=1}^n \mc{S}_i.
$$
The complement of $\mc{S}$ is then given by 
$$
	\mc{S}^c = \bigcap_{i=1}^n \mc{S}_i^c = \bigcap_{i=1}^n \overline{\mc{S}_i}.
$$
Set $\mc{T}_i = \mc{S} - \mc{S}_i$ 
for $1 \le i \le n$, and let 
$$
	\mc{T}= \mc{S} - \bigcap_{i=1}^n \mc{S}_i = \bigcup_{i=1}^n \mc{T}_i.
$$
Let 
$$
	\ell = \rank_{\Lambda} Y_{\mc{S}}^{\psi} = d-d_{\mc{S}^c},
$$ 
and note that $\ell = \rank_{\Lambda} I_{\mc{T}_i}^{\psi}$ for all $i$ by \eqref{eq:rankID}.
Recall that $\mc{L}_{\mc{S}_i,\psi} \in \Lambda$ is 
taken to be an element satisfying
$c_1(X_{\mc{S}_i}^{\psi}) = (\mc{L}_{\mc{S}_i,\psi})$.
 
We have that $r_{\mf{p}} = d_{\mf{p}} + 1 \ge 2$ for each $\mf{p} \in S_f$ by Lemma \ref{lem:sweepup}. Thus, by Remarks \ref{rem:conseqD=I}
and \ref{EKp}, for every $P \subset S_f$ we have
\begin{itemize}
	\item $I_P = D_P$, 
	\item $X_P \to X^{\flat}_P$ is an isomorphism, 
	\item $\mc{K}_P$ is supported in codimension $\min\{r_{\mf{p}} \mid \mf{p} \in P\}$, and
	\item $X_P \to Y_P$ is an injective pseudo-isomorphism.
\end{itemize}
We will use these facts without further reference.

Since we next work with eigenspaces that are $\Lambda$-modules, it is useful to compare their support with those of the 
original $\Omega$-modules.  For this, we have the following remark.  

\begin{remark} \label{rem:eigsp}
    Since $\Omega \cong \zp\ps{\Gamma}[\Delta]$ and $\Delta$ is of prime-to-$p$ order, 
    every prime ideal of $\Omega$ is the inverse image of a prime ideal of the quotient 
    $\Omega \cdot e_{\psi} \cong \mc{O}_{\psi}\ps{\Gamma}$
    for an idempotent $e_{\psi} \in \zp[\Delta]$ arising from the $G_{\qp}$-conjugacy class of a $p$-adic character $\psi$ of $\Delta$,
    where $\mc{O}_{\psi}$ denotes the $\zp$-algebra generated by the values of $\psi$.  
    
    Let us show that a prime $\mf{q}$ of $\Lambda$ is in the support of 
    $M^\psi = M \hat{\otimes}_{\zp[\Delta],\psi} W$ for a finitely generated $\Omega$-module $M$ if and only if the inverse image of
    $\mf{q}' = \mf{q} \cap \mc{O}_{\psi}\ps{\Gamma}$ in $\Omega$ is in the support of $M$.  This will allow us to apply the results 
    of Section \ref{s:duality} to study the $\Delta$-eigenspaces of our arithmetically-interesting $\Omega$-modules, as we shall do below.
    
    Let $N$ be the $\mc{O}_{\psi}\ps{\Gamma}$-module $M \hat{\otimes}_{\zp[\Delta],\psi} \mc{O}_{\psi}$, so that $M^\psi = N_W$
    for $N_W = N \hat{\otimes}_{\mc{O}_{\psi}} W$.  It will suffice to show $\mf{q}$ is in the support of $N_W$ if and only if $\mf{q}'$
    is in the support of $N$.  Let $F' \to F \to N \to 0$ be an exact sequence of $\mc{O}_{\psi}\ps{\Gamma}$-modules in which
    $F$ and $F'$ are  free of finite ranks $r$ and $s$, respectively.  This sequence defines an $r\times s$ presentation matrix
    $B$ after choosing bases for $F$ and $F'$.  The prime $\mf{q}'$ is not 
    in the support of $N$ if and only if some maximal minor of $B$ has determinant not in $\mf{q}'$.  Taking completed tensor products over $W$
    is a right exact functor on pseudo-compact $\mc{O}_\psi$-modules by \cite[Sect.~0.3.2]{Gabriel}, so $(F')_W \to F_W \to N_W \to 0$ is exact.
    It follows that $\mf{q}$ is not in the support of $N_W$ if and only if some maximal minor of $B$ has determinant which is not in $\mf{q}$.
    Our claim is now clear since the determinants of all the maximal minors lie in $\mc{O}_{\psi}\ps{\Gamma}$, and $\mf{q}' = \mf{q} \cap \mc{O}_{\psi}\ps{\Gamma}$.
    \end{remark}

We may now state and prove our first main theorem.

\begin{theorem} \label{thm:first_big_thm}
	Let $\mf{q}$ be a prime of $\Lambda$ not in the support of 
	$$
		(X_{\mc{S}^c}^{\omega\psi^{-1}})^{\iota}(1) \oplus (\mc{K}_{\mc{S},0}^{\omega\psi^{-1}})^{\iota}(1) 
		\oplus \mc{K}_{\overline{\mc{S}},0}^{\psi}.
	$$  
	Then we have an isomorphism of $\Lambda_{\mf{q}}$-modules
	$$
		\frac{\bigwedge^\ell X_{\mc{S},\mf{q}}^{\psi}}{\bigwedge^\ell I_{\mc{T}_1,\mf{q}}^{\psi}  
   	 	+ \cdots + \bigwedge^\ell I_{\mc{T}_n,\mf{q}}^{\psi} } \cong
		\frac{\Lambda_{\mf{q}}}{(\mc{L}_{\mc{S}_1,\psi},\ldots,\mc{L}_{\mc{S}_n,\psi})}.
	$$
\end{theorem}

\begin{proof}  
Let $\mf{q}$ be a prime of $\Lambda$.
If $X_{\mc{S},\mf{q}}^{\psi}$ is free, then we have an isomorphism
$\bigwedge^{\ell} X_{\mc{S},\mf{q}}^{\psi} \cong \Lambda_{\mf{q}}$.  If $I_{\mc{T}_i,\mf{q}}^{\psi}$ is free,
then since
$c_1(X_{\mc{S},\mf{q}}^{\psi}/I_{\mc{T}_i,\mf{q}}^{\psi}) = (\mc{L}_{\mc{S}_i,\psi})$, 
 this isomorphism
takes the free rank one submodule $\bigwedge^{\ell} I_{\mc{T}_i,\mf{q}}^{\psi}$ to 
$(\mc{L}_{\mc{S}_i,\psi})$.   So, we need only avoid those $\mf{q}$ such that $X_{\mc{S},\mf{q}}^{\psi}$
or some $I_{\mc{T}_i,\mf{q}}^{\psi}$ is not free.

 By Theorem \ref{thm:free} (noting Remark \ref{rem:eigsp}), the module $Y_{\mc{S},\mf{q}}^{\psi}$ is free for $\mf{q}$ outside
 the support of $(Y_{\mc{S}^c}^{\omega \psi^{-1}})^\iota (1) \oplus Z_{\mc{S}}^\psi$, with $Z_{\mc{S}}$ 
 as in \eqref{eq:Z}.
Lemma \ref{lem:needlater} provides an exact sequence
$$0 \to (X_{\mc{S}^c}^{\omega \psi^{-1}})^\iota (1)  \to (Y_{\mc{S}^c}^{\omega 
\psi^{-1}})^\iota (1)  \to (\mc{K}_{\mc{S},0}^{\omega \psi^{-1}})^\iota (1)  \to 0.$$
So, $Y_{\mc{S},\mf{q}}^{\psi}$ is free for $\mf{q}$ not in the support of 
$(X_{\mc{S}^c}^{\omega \psi^{-1}})^\iota (1)  \oplus (\mc{K}_{\mc{S},0}^{\omega \psi^{-1}})^{\iota}(1) \oplus Z_{\mc{S}}^\psi$.
Similarly, the homomorphism $X_{\mc{S},\mf{q}}^{\psi} \to Y_{\mc{S},\mf{q}}^{\psi}$ is an isomorphism 
for $\mf{q}$ not in the support of $\mc{K}_{\mc{S}^c,0}^\psi$ by Lemma \ref{lem:needlater}.
Finally, Corollary \ref{cor:localfree} tells us that every $I_{\mc{T}_i,\mf{q}}^{\psi}$ is free for
$\mf{q}$ not in the support of $\mc{K}_{\mc{T}}^\psi \oplus (\mc{K}_{\mc{T}}^{\omega \psi^{-1}})^{\iota}(1)$.

Together, the above conditions say that the desired isomorphism holds if we avoid primes in the support of
\begin{equation} \label{eq:toomuch}
	(X_{\mc{S}^c}^{\omega \psi^{-1}})^\iota (1)  \oplus (\mc{K}_{\mc{S},0}^{\omega \psi^{-1}})^{\iota}(1) \oplus Z_{\mc{S}}^\psi
	\oplus \mc{K}_{\mc{S}^c,0}^\psi \oplus \mc{K}_{\mc{T}}^\psi \oplus (\mc{K}_{\mc{T}}^{\omega \psi^{-1}})^{\iota}(1).
\end{equation}
This may be simplified to the statement of the theorem by the following observations.  If $n = 1$, then $\mc{S} \neq S_f$,
so $Z_{\mc{S}}^{\psi} = 0$, and $\mc{T} = \varnothing$, so $\mc{K}_{\mc{T}} = 0$.  Moreover $\mc{S}^c = \overline{\mc{S}}$ in this
case.  If $n \ge 2$, then note that $\overline{\mc{S}} = \mc{S}^c \cup \mc{T}$ and $\mc{T} \subset \mc{S}$.  
Both $\mc{S}$ and its conjugate set
$\overline{\mc{S}}$ have more than one element. 
This implies that $\zp^{\psi}$ is a subquotient
of $\mc{K}_{\overline{\mc{S}},0}^{\psi}$ and $\zp^{\omega\psi^{-1}}(1)$ is a subquotient
of $(\mc{K}_{\mc{S},0}^{\omega \psi^{-1}})^{\iota}(1)$. In turn, these two facts yield  that the supports of the third, fourth, and fifth terms 
in \eqref{eq:toomuch} are contained in the support of $\mc{K}_{\overline{\mc{S}},0}^{\psi} = \ker(\mc{K}_{\mc{T}}^{\psi} \oplus \mc{K}_{\mc{S}^c}^{\psi} \to \zp^{\psi})$, and the support of the last term is contained in the support of the second.
\end{proof}

\begin{remark}
	Regarding the disallowed primes in Theorem \ref{thm:first_big_thm}, note that
	$$
		(\mc{K}_{\mc{S}}^{\omega\psi^{-1}})^{\iota}(1) \cong \mc{K}_{\mc{S}}^{\psi\omega^{-1}}(1)
	$$
	as $\Lambda[\Delta]$-modules by Remark \ref{rem:noiota} (in fact, $\mc{K}_{\mc{S}}^{\psi\omega^{-1}} \cong
	\mc{K}_{\mc{S}}^{\omega\psi^{-1}}$ as $\Lambda$-modules as well), but we have written it as we have to exhibit
	a certain symmetry.
\end{remark}

The following notation is used in the statements of the various theorems in this section.

\begin{definition}
\label{def:Us}
Let  $\mc{U}_{\mf{p},\psi}$ (resp. $\overline{\mc{U}}_{\mf{p},\psi}$) denote the set of codimension two primes of $\Lambda$ in the support of $\mc{K}_{\mf{p}}^{\psi}$ 
 (resp. $(\mc{K}_{\mf{p}}^{\omega\psi^{-1}})^{\iota}(1)$).  For all subsets $\Sigma$ of $
 S_f$, let $$\mc{U}_{\Sigma,\psi} = \bigcup_{\mf{p} \in \Sigma} \ \mc{U}_{\mf{p},\psi} \quad \mathrm{and}\quad 
\overline{\mc{U}}_{\Sigma,\psi} = \bigcup_{\mf{p} \in \Sigma} \ \overline{\mc{U}}_{\mf{p},\psi}.$$ Define $\mc{Z}_{\Sigma,\psi}$ to be the free abelian group on $\mc{V}_{\Sigma,\psi} =  \mc{U}_{\Sigma^c,\psi} \cup \overline{\mc{U}}_{\Sigma,\psi}$, 
which we view a direct summand
of the free abelian group on the codimension two primes of $\Lambda$.
\end{definition}

\begin{remark}
	\label{rem:Us}
	By the discussion of Section \ref{s:CM}, 
	the set $\mc{U}_{\mf{p},\psi}$ is nonempty if and only if $r_{\mf{p}} = 2$ and $\psi|_{\Delta_{\mf{p}}} = 1$, in which 
	case $\mc{K}_{\mf{p}}^{\psi}$ is isomorphic to $W\ps{\Gamma/\Gamma_{\mf{p}}}$.  Similarly, 
	$\overline{\mc{U}}_{\mf{p},\psi} \ne \varnothing$ if and only if $r_{\mf{p}} = 2$ and $\omega\psi^{-1}|_{\Delta_{\mf{p}}} = 1$, 
	in which case $(\mc{K}_{\mf{p}}^{\omega\psi^{-1}})^{\iota}(1)$ is isomorphic to $W\ps{\Gamma/\Gamma_{\mf{p}}}(1)$.
	
	The groups $\Gamma_{\mf{p}} \otimes_{\zp} \qp$ and $\Gamma_{\overline{\mf{p}}} \otimes_{\zp} \qp$ are the same inside 
	$\Gamma \otimes_{\zp} \qp$ 
	if $\mf{p}$ and $\overline{\mf{p}}$ are conjugate primes in $S_f$. 
	For any CM type $\Sigma$, we have 
	$$
		\Gamma^- \otimes_{\zp} \qp = \bigoplus_{\mf{p} \in \Sigma} (\Gamma^-)_{\mf{p}} \otimes_{\zp} \qp
	$$
	from the proof of Lemma \ref{lem:sweepup}.
	Thus, if $\mf{p}$ and $\mf{p}'$ are distinct, non-conjugate primes, then $\Gamma_{\mf{p}} \cap \Gamma_{\mf{p}'}$
	has rank at most one and $r \ge 3$, so $\mc{U}_{\mf{p},\psi} \cap \mc{U}_{\mf{p}',\psi} = \varnothing$ and 
	$\overline{\mc{U}}_{\mf{p},\psi} \cap \overline{\mc{U}}_{\mf{p}',\psi} = \varnothing$.  Since
	$\Gamma_{\mf{p}}$ acts trivially on $W\ps{\Gamma/\Gamma_{\mf{p}}}$ and via the $p$-adic cyclotomic character on
	$ W\ps{\Gamma/\Gamma_{\mf{p}}}(1)$, we have that $\mc{U}_{\mf{p},\psi} \cap \overline{\mc{U}}_{\mf{p}',\psi} = \varnothing$ 
	for all $\mf{p}, \mf{p}' \in S_f$, as can also be seen from Remark \ref{rem:explicit}.
\end{remark}

The following theorem is an extension of Theorem A without its assumption on $\psi$.  
In Theorem \ref{thm:unconditionalBIG} below, we will provide a more general
result in which we eliminate the appearance of $\mc{Z}_{\mc{S},\psi}$ at the cost of introducing kernels and cokernels of maps  between pseudo-null modules which are difficult to compute explicitly.  

\begin{theorem}  
\label{thm:unconditional}
Any generator $\theta_0$ of $t_1(X_{\mc{S}}^\psi)$ divides any gcd $\theta$ in $\Lambda$ of 
$\mc{L}_{\mc{S}_1,\psi}, \ldots, \mc{L}_{\mc{S}_n,\psi}$,  
and we have a congruence of second Chern classes 
\begin{multline}
\label{eq:c2formulaquartic}
c_2\left(\frac{\Lambda }{(\mc{L}_{\mc{S}_1,\psi}/\theta, \ldots, \mc{L}_{\mc{S}_n,\psi}/\theta)} \right)   \equiv
t_2\left( \frac{(\bigwedge^\ell  X_{\mc{S}}^\psi)_{\tf}}{(\bigwedge^\ell  I_{\mc{T}_1}^\psi)_{\tf}  + \cdots + (\bigwedge^\ell  I_{\mc{T}_n}^\psi)_{\tf} }\right )  
\\
+ c_2\left (\frac{\theta}{\theta_0} \cdot \frac{\Lambda}{\Fitt(\rE^2(X_{\mc{S}^c}^{ \omega\psi^{-1} })(1))}\right ) \bmod \mc{Z}_{\mc{S},\psi}.
\end{multline}
\end{theorem}

\begin{proof}
To match the notation of Section \ref{sec:exterior} and Lemma \ref{lem:torsion!},
let $R$ be the localization of $\Lambda$ at a codimension two prime $\mf{q}$ not in $\mc{V}_{\mc{S},\psi}$, and set
$\mc{X} = X_{\mc{S},\mf{q}}^{\psi}$ and $\mc{F} = (X_{\mc{S},\mf{q}}^{\psi})^{**}$.  
Since $\mf{q} \notin \mc{U}_{\mc{S}^c,\psi}$,
Lemma \ref{lem:needlater} tells us that the injection $X_{\mc{S},\mf{q}}^{\psi} \to Y_{\mc{S},\mf{q}}^{\psi}$  is an isomorphism.  Similarly, since $\mf{q} \notin \overline{\mc{U}}_{\mc{S},\psi}$, we have that
$$
	\rE^2(X_{\mc{S}^c}^{\omega\psi^{-1}})(1)_{\mf{q}} \to  \rE^2(Y_{\mc{S}^c}^{\omega\psi^{-1}})(1)_{\mf{q}}
$$ 
is an isomorphism.
By Proposition \ref{prop:localseq}, we then have $\mc{E} = \rE^2(X_{\mc{S}^c}^{\omega\psi^{-1}})(1)_{\mf{q}}$, so $\theta_1$ is a unit. 
Moreover, $\rQ(\mc{E})$ is pseudo-null as the cokernel of the map from $\bigwedge^{\ell} \mc{X}$ to its
reflexive hull.

We also set $\mc{I}_i = I_{\mc{T}_i,\mf{q}}^{\psi}$ and $\mc{J}_i=(I_{\mc{T}_i,\mf{q}}^{\psi} )^{**}$.
The canonical maps $\mc{I}_i \to \mc{J}_i$ are isomorphisms of free $\Lambda_{\mf{q}}$-modules by Corollary \ref{cor:localfree} since 
$\mf{q} \notin \overline{\mc{U}}_{\mc{T},\psi}$.  
We may therefore identify the image $(\bigwedge^\ell \mc{I}_i)_{\tf}$ of
$\bigwedge^\ell\mc{I}_i$ in $\bigwedge^\ell \mc{X}$ with $\bigwedge^\ell\mc{I}_i$.
As $\mc{B}_i=0$ in the notation of Lemma \ref{lem:torsion!}, the result follows from the short exact sequence 
(\ref{eq:inertiathree}) in Corollary \ref{cor:pseudo}.
\end{proof}

\begin{corollary}
\label{cor:vanish}  
If $n=2$ and $\mc{V}_{\mc{S},\psi} = \varnothing$, then the following are equivalent.
\begin{enumerate}
\item[(i)] The class $c_2\left(\frac{\Lambda}{(\mc{L}_{\mc{S}_1,\psi}/\theta,\mc{L}_{\mc{S}_2,\psi}/\theta)} \right) $
on the left-hand side of \eqref{eq:c2formulaquartic} is trivial.
\item[(ii)] One of $\mc{L}_{\mc{S}_1,\psi}$ and $\mc{L}_{\mc{S}_2,\psi}$ divides the other, so  
$$(\mc{L}_{\mc{S}_1,\psi}/\theta,\mc{L}_{\mc{S}_2,\psi}/\theta) = \Lambda.$$ 
\item[(iii)]  We have
\begin{eqnarray*}
t_2 \left ( \frac{(\bigwedge^\ell  X_{\mc{S}}^\psi)_{\tf}}{\bigwedge^\ell  I_{\mc{T}_1}^\psi  + \bigwedge^\ell  I_{\mc{T}_2}^\psi }\right ) = 0
&\text{and}&
c_2\left (\frac{\theta}{\theta_0} \cdot \frac{\Lambda}{\Fitt(\rE^2(X_{\mc{S}^c}^{ \omega\psi^{-1} })(1))}\right ) = 0.
\end{eqnarray*}
\end{enumerate}
\end{corollary}

\begin{proof} The equivalence of (i) and (ii) follows from \cite[Lem.~A.3]{BCGKPST}.  The fact that (ii) and (iii) are equivalent follows
from the fact that the length of the localization of a module at a prime is a nonnegative integer when this localization has finite length.
\end{proof}

\begin{remark} 
\label{rem:divisorrem}
We suspect that the greatest common divisor $\theta$ in Corollary \ref{cor:vanish} is sometimes nontrivial. To be precise, we believe that this may happen if $\psi$ satisfies the condition $\psi \cdot (\psi \circ j) = \omega$, where $j$ is the involution of $\Gal(E^{\ab}/E)$ given by 
conjugating by any lift of the generator of $\Gal(E/E^+)$. The nontrivial $\theta$ should be $\Theta=\gamma_{\mr{cyc}} - \sqrt{\chi_p(\gamma_{\mr{cyc}})}$, where $\gamma_{\mr{cyc}}$ is a topological generator for $\Gamma^{+}$ and $\chi_p$ is the $p$-power cyclotomic character. (In this remark, we assume the validity of Leopoldt's  conjecture for $E$  so that $\Gamma^{+}$ is topologically cyclic.) Note that $\chi_p(\gamma_{\mr{cyc}})$ is a principal unit and the square root should be  chosen to be a principal unit. There exist  continuous characters $\Psi$ of $\mathcal{G}$ satisfying the conditions
$$\Psi \big|_{\Delta} = \psi, \quad \Psi \cdot (\Psi \circ j) =\chi_p.$$
We have
$\Psi(\gamma_{\mr{cyc}})=\sqrt{\chi_p(\gamma_{\mr{cyc}}) }$  for any such $\Psi$ and hence $\Psi(\Theta)=0$. Conversely, $\Psi(\Theta) =0$ implies that $\Psi \cdot (\Psi \circ j) =\chi_p$. 
Let $\Sigma$ be any  CM type, and let  $\mr{L}_{\Sigma,\psi} \in \Lambda$ be the Katz $p$-adic $L$-function
attached to $\Sigma$ and $\psi$.  (This $L$-function is given up to a certain power of $p$ by integrating the inverse of a character against the Katz measure.)
It follows that $\Theta$ divides $\mr{L}_{\Sigma, \psi}$ if and only if $\Psi\big( \mr{L}_{\Sigma, \psi}\big)=0$  for all $\Psi$ satisfying the above conditions. In fact, if $\Psi_0$ is one such $\Psi$, it is sufficient to have $\Psi\big( \mr{L}_{\Sigma, \psi}\big)=0$  for all $\Psi$ of the form $\Psi=\Psi_0 \cdot \rho$, where $\rho$ is a character of $\Gamma^-$ of finite order. 

It is possible to choose $\Psi_0$ to be the Galois character attached to a Gr\"ossencharacter of type $A_0$ for $E$ whose infinity type lies in the interpolation range for $\mr{L}_{\Sigma, \psi}$. The corresponding complex $L$-function will have a functional equation relating that $L$-function to itself. If the  sign in that functional equation is $-1$, then the central critical value will be forced to vanish. The same thing will be true for $\Psi=\Psi_0 \cdot \rho$ for any finite order character $\rho$ of $\Gamma^-$.  That would mean that $\Psi\big( \mr{L}_{\Sigma, \psi}\big)=0$ for such $\Psi$ if the corresponding sign is $-1$.  Now it turns out that for a given $\Sigma$ and $\psi$, the signs will be constant, either all $+1$ or all $-1$. We suspect that each sign will occur for half of the  CM types, possibly under some extra assumptions on $\psi$ and $E$. Therefore, assuming this is the case, if there are at least four $p$-adic CM-types for $E$, then at least two will have the corresponding signs equal to $-1$. Hence the corresponding $p$-adic $L$-functions will both be divisible by $\Theta$. Thus, examples where $\theta$ is nontrivial may possibly occur when $E$ has at least four primes above $p$.

An illustration of the kind of behavior described above can be found in \cite{GreenInvent}. That paper considers a case where $E$ is an imaginary quadratic field in which $p$ splits. Note however that there are just two primes above $p$ in that case, and it is proved that $\Theta$ is actually not a common divisor of the two $p$-adic $L$-functions. 
\end{remark}

The following result provides a more general version of Theorem \ref{thm:unconditional} that avoids working modulo
$\mc{Z}_{\mc{S},\psi}$ at the expense of a longer statement that includes a new ``error term'' $c_2(C_{\mc{S},\psi})$.

\begin{theorem}
\label{thm:unconditionalBIG}
Let $\theta_0$ be a generator of $t_1(Y_{\mc{S}}^\psi)$, which divides a gcd $\theta$ of 
$\mc{L}_{\mc{S}_1,\psi}, \ldots, \mc{L}_{\mc{S}_n,\psi}$.  
 Let $g \colon \Lambda^n \to \Lambda$ be given by
  $$
 	g( (\alpha_i)_i ) = \sum_{i=1}^n \frac{\mc{L}_{\mc{S}_i,\psi}}{\theta_0} \alpha_i,
$$
 and let $C_{\mc{S},\psi}$ be the cokernel of the map
 $$
 	\ker(\Lambda^n \xrightarrow{g} \Lambda) \to 
	\ker\left(\bigoplus_{i=1}^n \frac{\Lambda}{\Fitt(\rE^2(\mc{K}_{\mc{T}_i}^{\omega\psi^{-1}})(1))} 
	\xrightarrow{g'} \frac{\Lambda}{\Fitt(\rE^2(Y_{\mc{S}^c}^{\omega\psi^{-1}})(1))} \right)
 $$
 induced by the canonical quotient map, where $g'$ is the map induced by $g$.
There is an equality of second Chern classes of pseudo-null modules
 \begin{multline*}
 c_2 \left ( \frac{\Lambda}{(\mc{L}_{\mc{S}_1,\psi}/\theta,\ldots,\mc{L}_{\mc{S}_n,\psi}/\theta)}  \right ) =
 t_2\left ( \frac{(\bigwedge^\ell Y_{\mc{S}}^{\psi})_{\tf}}{  (\bigwedge^\ell I_{\mc{T}_1}^{\psi})_{\tf}+ \cdots + (\bigwedge^\ell I_{\mc{T}_n}^{\psi})_{\tf}} \right ) - c_2(C_{\mc{S},\psi}) \\
 + c_2 \left (\frac{\theta}{\theta_0} \cdot \frac{\Lambda}{\Fitt(\rE^2(Y_{\mc{S}^c}^{\omega\psi^{-1}})(1)) + (\mc{L}_{\mc{S}_1,\psi}/\theta_0)\Lambda + 
 \cdots + (\mc{L}_{\mc{S}_n,\psi}/\theta_0)\Lambda}\right ). 
 \end{multline*}
\end{theorem}

\begin{proof}
Let $\mf{q}$ be a codimension $2$ prime of $\Lambda$.  Then the localization $Y_{\mc{S},\mf{q}}^{**}$ is free 
as a reflexive module over the local ring $\Lambda_{\mf{q}}$ of Krull dimension $2$.   Note that $(\zp)_{\mf{q}} = 0$ if $r \ge 3$, and the map
$\rE^2(Y_{\mc{S}^c})(1) \to Z_{\Sigma}$ in Proposition \ref{prop:localseq} is zero if $r = 2$ by 
\cite[Prop.~4.1.17]{BCGKPST}.  
Lemma \ref{lem:ranklemma} gives the injectivity of $I_{\mc{T}_i}^{\psi} \to Y_{\mc{S}}^{\psi}$, so
we are by Proposition \ref{prop:localseq} in the situation of Lemma \ref{lem:torsion!} with 
\begin{eqnarray*}
    &R=\Lambda_{\mf{q}}, \quad \mc{X}=Y_{\mc{S},\mf{q}}^\psi, \quad \mc{F}=(Y_{\mc{S},\mf{q}}^\psi)^{**}, \quad
    \mc{E}=\rE^2(Y_{\mc{S}^c}^{\omega\psi^{-1}})(1)_{\mf{q}},&\\
   & \mc{I}_i=I_{\mc{T}_i,\mf{q}}^{\psi}, \quad \mc{J}_i=(I_{\mc{T}_i,\mf{q}}^{\psi})^{**}, \quad \mr{and} \quad
    \mc{B}_i=\rE^2(\mc{K}_{\mc{T}_i}^{\omega \psi^{-1}})(1)_{\mf{q}}.&
\end{eqnarray*}
Theorem \ref{thm:unconditionalBIG} then follows from Corollary \ref{cor:pseudo}, with Remark \ref{rem:pseudo} providing
the term $c_2(C_{\mc{S},\psi})$.
\end{proof}

We have $\ell = 1$ in Theorem \ref{thm:unconditional} if and only if $n = 2$ and the CM types $\mc{S}_1$ and $\mc{S}_2$ differ by only one prime, which is of degree $1$ (i.e., $r_{\mf{p}} = 2$).  In this case, we obtain the following more explicit results.  In particular, Proposition \ref{prop:pn} and Theorem \ref{thm:rank1}  imply Theorem  C.

\begin{proposition}
\label{prop:pn}
	Suppose that $\ell = 1$ so that $n = 2$.  The following conditions are equivalent:
	\begin{itemize}
	\item[(a)] $X_{\mc{S}_1 \cap \mc{S}_2}^\psi$ and $X_{\overline{\mc{S}}_1 \cap \overline{\mc{S}}_2}^{\omega \psi^{-1}}$ are 
	both pseudo-null,
	\item[(b)] $\mc{L}_{\mc{S}_1,\psi}$ and $\mc{L}_{\mc{S}_2,\psi}$ are relatively prime.   
	\end{itemize}
\end{proposition}

\begin{proof}
	Let 
	\begin{eqnarray*}
	\Sigma = \mc{S}_1 \cap \mc{S}_2 &\mr{and}& \overline{\Sigma} = \mc{S}^c = \overline{\mc{S}}_1 \cap \overline{\mc{S}}_2.
	\end{eqnarray*}
	Set $L_i = \mc{L}_{\mc{S}_i,\psi}$ for brevity.  As we have remarked,
	$X_{\Sigma} \to Y_{\Sigma}$ is injective with pseudo-null cokernel, so $X_{\Sigma}^{\psi}$
	is pseudo-null if and only if $Y_{\Sigma}^{\psi}$ is. Similarly, $X_{\overline{\Sigma}}^{\omega\psi^{-1}}$ is pseudo-null
	if and only if $Y_{\overline{\Sigma}}^{\omega\psi^{-1}}$ is.  
	
	Suppose that (b) holds.  In this case, since both $L_1$ and $L_2$ annihilate $X_{\Sigma}^{\psi}$ by definition and 
	are relatively prime by assumption, $X_{\Sigma}^{\psi}$ is pseudo-null.  
	We now conclude from 
	Proposition \ref{prop:localseq}  and \cite[Prop.~4.1.17]{BCGKPST} that there is a map of exact sequences
	\begin{equation} \label{eq:simplediag}
		\SelectTips{cm}{} \xymatrix{
		& 0 \ar[r] & I_{\mc{T}_i}^{\psi} \ar[r] \ar[d] & (I_{\mc{T}_i}^{\psi})^{**} \ar[d] \ar[r] & \rE^2(\mc{K}_{\mc{T}_i}^{\omega \psi^{-1}})(1) \ar[d] \ar[r] & 0\\
		0 \ar[r] & \rE^1(Y_{\overline{\Sigma}}^{\omega\psi^{-1}})(1) \ar[r]
		& Y_{\mc{S}}^{\psi} \ar[r] & (Y_{\mc{S}}^{\psi})^{**} \ar[r] & \rE^2(Y_{\overline{\Sigma}}^{\omega\psi^{-1}})(1)}
	\end{equation}
	for $i \in \{1,2\}$.
	The leftmost vertical map in \eqref{eq:simplediag} for a given $i$ has torsion cokernel with first Chern class 
	$c_1(X_{\mc{S}_i}^{\psi}) = (L_i)$.  This forces the map $(I_{\mc{T}_i}^{\psi})^{**} \to (Y_{\mc{S}}^{\psi})^{**}$ between
	free $\Lambda$-modules of rank one to be injective.
	From the diagram, we then see that
	the first Chern class of the torsion $\Lambda$-module
	$\rE^1(Y_{\overline{\Sigma}}^{\omega\psi^{-1}})(1)$ divides $(L_i)$.  Since $L_1$ and $L_2$
	are relatively prime, this forces $\rE^1(Y_{\overline{\Sigma}}^{\omega\psi^{-1}})$ to be pseudo-null, which 
	can only occur if the torsion module $Y_{\overline{\Sigma}}^{\omega\psi^{-1}}$ is pseudo-null.  Thus,
	$X_{\overline{\Sigma}}^{\omega\psi^{-1}}$ is pseudo-null as well.
	
	Now suppose that (a) holds.  
	We again use the diagram \eqref{eq:simplediag} but now have that the term $\rE^1(Y_{\overline{\Sigma}}^{\omega\psi^{-1}})(1)$ 
	is zero since $X_{\overline{\Sigma}}^{\omega\psi^{-1}}$ is pseudo-null.  
	Since $(I_{\mc{T}_i}^{\psi})^{**} \to (Y_{\mc{S}}^{\psi})^{**}$ is a map between free $\Lambda$-modules
	of rank $1$, we see that upon appropriate choices of $\Lambda$-bases it is given by multiplication by 
	$L_i$.  Applying the direct sum of the vertical maps in \eqref{eq:simplediag} for $i \in \{1,2\}$, we get a composite 
	map 
	$$
		X_{\Sigma}^{\psi} \to Y_{\mc{S}}^{\psi}/I_{\mc{T}_1 \cup \mc{T}_2}^{\psi} \to \Lambda/(L_1,L_2)
	$$ 
	on cokernels which is a pseudo-isomorphism by the snake lemma.  Since $X_{\Sigma}^{\psi}$
	is pseudo-null, so is $\Lambda/(L_1,L_2)$, and therefore $L_1$ and $L_2$ are relatively prime.
\end{proof}

\begin{remark} \label{rem:c2E2}
	We claim that $c_2(\rE^2(M)) = c_2(M^{\iota})$ for any finitely generated pseudo-null $\La$-module $M$.  Since 
	$\rE^2(M)^{\iota} = \Ext^2_{\La}(M,\La)$, we need only verify that 
	$$
		c_2(\Ext^2_{\La_P}(M_P,\La_P))=c_2(M_P)
	$$
	upon localization at a height $2$ prime $P$ of $\La$.  Since $\La_P$ is 
	regular of dimension $2$, the localization $M_P$ has a finite filtration with graded pieces isomorphic to $\La_P/P\La_P$ 
	(cf. \cite[Lem.~A.2]{BCGKPST}). 
	For any short exact sequence  $0\to N\to M_P\to \La_P/P\La_P \to 0$ of $\La_P$-modules, 
	we have $\Ext^1_{\La_P}(N,\La_P)=0$ since $N$ is pseudo-null, and $\Ext^3_{\La_P}(\La_P/P\La_P,\La_P)=0$ since $\Lambda_P$ 
	has dimension 2.  
	Since $\Ext^2_{\La_P}(\La_P/P\La_P,\La_P)=\La_P/P\La_P$ and second Chern classes are
	additive with respect to short exact sequences of pseudo-null modules, our claim now follows by induction.
\end{remark}

\begin{theorem}
\label{thm:rank1} 
	Let $\ell = 1$, and suppose that $X_{\mc{S}_1 \cap \mc{S}_2}^\psi$ and 
	$X_{\overline{\mc{S}}_1 \cap \overline{\mc{S}}_2}^{\omega \psi^{-1}}$ are both pseudo-null.  
	Then there is an equality of second Chern classes of pseudo-null modules
	 \begin{equation}
	\label{eq:c2formulaquarticone}
	c_2\left(\frac{\Lambda}{(\mc{L}_{\mc{S}_1,\psi},\mc{L}_{\mc{S}_2,\psi})} \right)  = 
	c_2(X_{\mc{S}_1 \cap \mc{S}_2}^\psi) + c_2( (X_{\overline{\mc{S}}_1 \cap \overline{\mc{S}}_2}^{ \omega\psi^{-1} })^\iota(1) ) 
	+ c_2(\mc{K}_{\overline{\mc{S}}_1 \cap \overline{\mc{S}}_2}^{\psi}) 
	+ c_2((\mc{K}_{\mc{S}_1 \cap \mc{S}_2}^{\omega\psi^{-1}})^{\iota}(1)).  
	\end{equation}
\end{theorem}

\begin{proof} If $E$ is imaginary quadratic, this is \cite[Thm.~5.2.5]{BCGKPST}, so we assume in what follows that $[E:\mathbb{Q}] > 2$.  
	As in the proof of Proposition \ref{prop:pn}, we let $\Sigma = \mc{S}_1 \cap \mc{S}_2$ and
	$\overline{\Sigma} = \mc{S}^c = \overline{\mc{S}}_1 \cap \overline{\mc{S}}_2$ and set $L_i = \mc{L}_{\mc{S}_i,\psi}$ for $i \in \{1,2\}$.
	Consider the set $\mc{T} = \mc{T}_1 \cup \mc{T}_2$ of cardinality $2$.
	The maps of \eqref{eq:simplediag} for $i \in \{1,2\}$ yield a diagram of exact sequences
	\begin{equation}
	\label{eq:diagram!}
	\SelectTips{cm}{} \xymatrix{
		0 \ar[r] & I_{\mc{T}}^{\psi} \ar[r] \ar[d]_{f_1} & (I_{\mc{T}}^{\psi})^{**} 
		\ar[d]_{f_2} \ar[r] & \rE^2(\mc{K}_{\mc{T}}^{\omega\psi^{-1}})(1) \ar[d]_{f_3} \ar[r] & 0\\
		0 \ar[r] & Y_{\mc{S}}^{\psi} \ar[r] & (Y_{\mc{S}}^{\psi})^{**} \ar[r] & \rE^2(Y_{\overline{\Sigma}}^{\omega\psi^{-1}})(1) \ar[r] & 0.
	}
	\end{equation}
	(Note that $Z_{\mc{S}} = 0$ since $\mc{S} \neq S_f$, so we have the right exactness in the lower row.)
	We show that $f_3$ is an injection up to modules supported in codimension greater than $2$, so 
	$$
		c_2(\coker(f_2)) = c_2(\coker(f_1)) + c_2(\coker(f_3)).
	$$
	
	From the exact sequence of Lemma \ref{lem:needlater} and the pseudo-nullity of 
	$X_{\overline{\Sigma}}^{\omega\psi^{-1}}$, we have an exact sequence of
	Ext-groups
	\begin{equation}
	\label{eq:Ext1}
		0  \to \rE^2(\mc{K}_{\mc{S},0}^{\omega\psi^{-1}}) \to \rE^2(Y_{\overline{\Sigma}}^{\omega\psi^{-1}})
		\to \rE^2(X_{\overline{\Sigma}}^{\omega\psi^{-1}}) \to \rE^3(\mc{K}_{\mc{S},0}^{\omega\psi^{-1}}).
	\end{equation}
	Since $r \ge 3$, the map  $\rE^2(\mc{K}_{\mc{S}}^{\omega\psi^{-1}}) \to \rE^2(\mc{K}_{\mc{S},0}^{\omega\psi^{-1}})$ is an injection, and 
	since $\overline{\Sigma}^c = \mc{S} = \mc{T} \cup \Sigma$, the group $\rE^2(\mc{K}_{\mc{S}}^{\omega\psi^{-1}})$ contains
	$\rE^2(\mc{K}_{\mc{T}}^{\omega\psi^{-1}})$ as a direct summand.  
	It follows that $f_3$ is an injection.  Since
	$\rE^3(\mc{K}_{\mc{S},0}^{\omega\psi^{-1}})$ is supported  in codimension greater than $2$, 
	using (\ref{eq:diagram!}) and  (\ref{eq:Ext1}), we obtain
	\begin{align*}
		c_2(\coker(f_3)) &= c_2(\rE^2(Y_{\overline{\Sigma}}^{\omega\psi^{-1}})(1)) - c_2(\rE^2(\mc{K}_{\mc{T}}^{\omega\psi^{-1}})(1))\\
		&= c_2(\rE^2(X_{\overline{\Sigma}}^{\omega\psi^{-1}})(1)) +  c_2(\rE^2(\mc{K}_{\mc{S}}^{\omega\psi^{-1}})(1)) - c_2(\rE^2(\mc{K}_{\mc{T}}^{\omega\psi^{-1}})(1))\\
		&=c_2(\rE^2(X_{\overline{\Sigma}}^{\omega\psi^{-1}})(1))+c_2(\rE^2(\mc{K}_{\Sigma}^{\omega\psi^{-1}})(1))\\
		&=c_2((X_{\overline{\Sigma}}^{\omega\psi^{-1}})^{\iota}(1))+ c_2((\mc{K}_{\Sigma}^{\omega\psi^{-1}})^{\iota}(1)),
	\end{align*}
	the last equality following from Remark 
	\ref{rem:c2E2}.
	As in the proof of Proposition \ref{prop:pn}, the cokernel of $f_2$ is pseudo-null with second Chern class 
	$$
		c_2(\coker(f_2)) = c_2(\Lambda/(L_1,L_2)).
	$$  
	The cokernel of $f_1$ is similarly pseudo-null by assumption, and it has second Chern class 
	$$
		c_2(\coker(f_1)) = c_2(X_{\Sigma}^{\psi}) + c_2(\mc{K}_{\mc{S}^c}^{\psi}) = c_2(X_{\Sigma}^{\psi}) 
		+ c_2(\mc{K}_{\overline{\Sigma}}^{\psi}).
	$$
	The result now follows.
\end{proof}

 \begin{remark} The last two terms in equation (\ref{eq:c2formulaquarticone}) give ``common trivial zeros in codimension 2" for $\mc{L}_{\mc{S}_1,\psi}$ and $\mc{L}_{\mc{S}_2,\psi}$. 
Here, by ``common zeros'', we mean codimension two points which
are in the support of the maximal pseudo-null submodule of $\Lambda/(\mc{L}_{\mc{S}_1,\psi},\mc{L}_{\mc{S}_2,\psi})$.  
To illustrate this, note that $T_1$ and $T_2$ in $\zp\ps{T_1,T_2}$ share a common zero at the point $(T_1,T_2) = (0,0)$, viewed as functions on the product of two $p$-adic open discs  of radius $1$ around the origin in $\overline{\mathbb{Q}}_p$.  This corresponds to the fact that $\zp\ps{T_1,T_2}/(T_1,T_2)$  is a non-trivial pseudo-null module supported on the codimension two prime $(T_1,T_2)$.
By ``trivial zeros", we mean arising from trivial zeros of the corresponding Katz $p$-adic $L$-functions, as in Remark \ref{trivzero}.

The common trivial zeros of codimension two arise from the triviality of characters on decomposition groups and are described by
Remark \ref{rem:Us}. That is, $\mc{K}_{\mf{p}}^{\psi}$ for $\mf{p} \in \overline{\mc{S}}_1 \cap \overline{\mc{S}}_2$ (resp., 
$(\mc{K}_{\mf{p}}^{\omega\psi^{-1}})^{\iota}(1)$ for $\mf{p} \in \mc{S}_1 \cap \mc{S}_2$) has nontrivial second Chern class if and only if
$\psi|_{\Delta_{\mf{p}}} = 1$ (resp., $\omega\psi^{-1}|_{\Delta_{\mf{p}}} = 1$) and $r_{\mf{p}} = 2$. For such a $\mf{p}$, the resulting second Chern class comes from the ideal determining the corresponding quotient in Remark \ref{rem:explicit}.
\end{remark}

\section{Canonical subquotients in the lower central series}
\label{s:grouptheory}

Let $\Pi$ be a profinite group.  
The lower central series of $\Pi$ is defined by $\Pi_0 = \Pi$,
and by letting $\Pi_i$ be the closure of  $[\Pi,\Pi_{i-1}]$ for $i \ge 1$. The maximal abelian quotient of $\Pi$ in the category
of profinite groups is
$\Pi^{\ab} = \Pi/\Pi_1$.  

We have a canonical commutator pairing
$$
\langle \ \, , \  \rangle \colon \Pi^{\ab} \times \Pi^{\ab} \to \Pi_1/\Pi_2
$$
defined on $x,y \in \Pi$ by
$$\langle \overline{x},\overline{y} \rangle = [x,y] \cdot \Pi_2,$$ 
where $[x,y] = xyx^{-1}y^{-1}$ and $\overline{x}$ is the image of $x$ in $\Pi^{\ab}$. 
(Note that $\Pi_1/\Pi_2$ is central in $\Pi/\Pi_2$, so this is well-defined.)
This is an alternating
pairing, and the image of the pairing generates all of $\Pi_1/\Pi_2$. 

Suppose $\Phi$ is a subgroup of  the group $\mr{Aut}(\Pi)$ of continuous automorphisms of $\Pi$.
Then $\Phi$ acts on all terms in the lower central series of $\Pi$.  The pairing 
$\langle \ \,,\ \rangle$ is equivariant for this action in the sense that 
\begin{equation*} \label{eq:equivariant}
\langle \sigma(\overline{x}),\sigma(\overline{y}) \rangle = \sigma(\langle \overline{x},\overline{y} \rangle ) \quad \mathrm{for} \quad \sigma \in \Phi.
\end{equation*}

The following lemma is clear. 

\begin{lemma}
\label{lem:Tadjoint} There is a largest quotient $(\Pi_1/\Pi_2)_{\Phi,\mathrm{s}}$ of $\Pi_1/\Pi_2$ by a $\Phi$-stable subgroup 
of the abelian group $\Pi_1/\Pi_2$ such that the pairing
$$
\langle \ \, , \ \rangle_{\Phi} \colon \Pi^{\ab} \times \Pi^{\ab} \to (\Pi_1/\Pi_2)_{\Phi,\mathrm{s}}
$$
is self-adjoint in the sense that
\begin{equation*} \label{eq:selfadj}
\langle \sigma(\overline{x}),\overline{y} \rangle_{\Phi} = \langle \overline{x},\sigma(\overline{y})\rangle_{\Phi}
\end{equation*}
for all $\sigma \in \Phi$ and $\overline{x},\overline{y} \in \Pi^{\ab}$.  \end{lemma}

\begin{remark}
\label{ex:niceex}  We add an $\mathrm{``s"}$ to the subscript so that there is no confusion of $(\Pi_1/\Pi_2)_{\Phi,\mathrm{s}}$
with the coinvariants of $\Phi$ acting on $\Pi_1/\Pi_2$.
Suppose that $\Pi$ is a closed normal subgroup of a profinite group $\tilde{\Pi}$.  The conjugation action of
$\tilde{\Pi}$ on $\Pi$ gives a subgroup $\Phi$ of $\mr{Aut}(\Pi)$ to which one can apply Lemma \ref{lem:Tadjoint}.  
\end{remark}

The following result is a topological variant on exercises in \cite{Brown}. 
The key ingredient is the universal coefficient theorem for group homology and group cohomology;  see 
\cite[Exercise 3, \S III.1]{Brown}.

\begin{proposition}  
\label{prop:Brown}
Let $H$ be an abelian pro-$p$ group acting trivially on a discrete abelian group $A$.  
Let $H \wedge_{\zp} H$ be the (completed) wedge product of $H$ with itself in the
category of abelian pro-$p$ groups.  
Then there is an exact sequence
\begin{equation} \label{eq:exact1}
0 \to \Ext^1(H,A) \to \rH^2(H,A) \xrightarrow{\theta} \Hom(H \wedge_{\zp} H,A) \to 0
\end{equation}
of abelian groups defined in the following way, where here $\Hom$ and $\Ext^1$ are taken in the category of 
topological abelian groups.  Each class in $\rH^2(H,A)$ is represented
by a continuous two cocycle $f \colon H \times H \to A$ normalized so that $f(0,h) = f(h,0) = 0$ for all $h \in H$.
The class $[f] \in \rH^2(H,A)$ is sent by $\theta$ to the homomorphism
$c \in \Hom(H \wedge_{\zp} H,A)$ defined by $c(h_1 \wedge h_2) = f(h_1,h_2) - f(h_2,h_1)$.
Moreover, suppose that
$$0 \to A \to \tilde{H} \to H \to 0$$
is a central extension of groups with class represented by $f$.  The function $c$ is given
by $c(h_1 \wedge h_2) = [\tilde{h}_1,\tilde{h}_2]$ for any lifts $\tilde{h}_1$ and $\tilde{h}_2$ of $h_1$ and $h_2$ to $\tilde{H}$.
\end{proposition} 

\begin{proof}
The map $\theta$ is the topological version of the map defined in Exercise 8 of \S IV.3 of \cite{Brown}.  In part (c) of this exercise, the
kernel of $\theta$ is identified with $\Ext^1(H,A)$.  The
steps involved in showing that  (\ref{eq:exact1}) is exact are outlined in Exercise 5 of \S V.6 of \cite{Brown}.  
 \end{proof}

For the remainder of this section, $G$ will be a profinite group and $\Pi$ will be its maximal pro-$p$ quotient. Let 
 $X = \Pi^{\ab}$ be the maximal abelian, pro-$p$ quotient of $\Pi$. 
Applying Proposition \ref{prop:Brown} in this context, we get a surjective homomorphism
\begin{equation}
\label{eq:thetanice}
\theta_X \colon \rH^2(X,\qp/\zp) \to \Hom(X \wedge_{\zp} X,\qp/\zp), 
\end{equation}
and the kernel of $\theta_X$ is the set of $[f] \in \rH^2(X,\qp/\zp)$ which represent abelian group
extensions of $X$ by $\qp/\zp$.
Let us take $B = \ker(G \to X)$, which is a closed subgroup of $G$.  
We have the Hochschild-Serre spectral sequence 
\begin{equation}
\label{eq:spectral}
E_2^{i,j}= \rH^i(X,\rH^j(B,\qp/\zp)) \Rightarrow \rH^{i+j}(G,\qp/\zp).
\end{equation}

\begin{lemma}
\label{lem:isom}
Suppose that $\rH^2(G,\qp/\zp) = 0$.
Both $\theta_X$ and the transgression map
$$
	\Tra \colon \Hom(B,\qp/\zp)^X \to \rH^2(X,\qp/\zp)
$$
are isomorphisms, yielding a composite isomorphism 
\begin{equation} \label{eq:thetanicer}
\Hom(B,\qp/\zp)^X \xrightarrow{\sim} \Hom(X \wedge_{\zp} X,\qp/\zp).
\end{equation}
\end{lemma}

\begin{proof}  
The spectral sequence \eqref{eq:spectral} and the triviality of $\rH^2(G,\qp/\zp)$ gives a four-term exact sequence
of base terms
\begin{multline*}
	0 \to \Hom(X,\qp/\zp) \xrightarrow{\Inf} \Hom(G,\qp/\zp) \\ 
	\xrightarrow{\Res} \Hom(B,\qp/\zp)^X \xrightarrow{\Tra} \rH^2(X,\qp/\zp) \to 0.
\end{multline*}
The inflation map $\Inf$
is surjective as $\qp/\zp$ is a direct limit of $p$-groups and
$X$ is the maximal abelian pro-$p$ quotient of $\Pi$.   
Thus $\Tra$ is an isomorphism.

We know from Proposition \ref{prop:Brown} that  $\theta_X$ is surjective.  Since $\Tra$ is an isomorphism, we may write any 
element in the kernel of $\theta_X$ as $\Tra(\phi)$ for some $\phi \in \Hom(B,\qp/\zp)^X$.
Then $\ker(\phi)$ is a subgroup of $B$ such that $B/\ker(\phi) \cong \im(\phi)$ is a
finite cyclic $p$-group.  We have a central extension of pro-$p$ groups
\begin{equation}
\label{eq:niceext}
1 \to \im(\phi) \to G/\ker(\phi) \to X \to 1
\end{equation}
since $G/B = X$ and $\phi$ is fixed by $X$. This extension provides the class of $-\!\Tra(\phi)$ (see \cite[Lemma 1.1]{sharifi-thesis}).
By Proposition \ref{prop:Brown} and the discussion which follows it,
the statement that $\theta_X(\Tra(\phi)) = 0$ is equivalent to the statement
that $G/\ker(\phi)$ is an abelian group.  However, $G /\ker(\phi)$
is then an abelian quotient of $\Pi$, and $X$
is the maximal abelian quotient of $\Pi$.  This proves that $B/\ker(\phi) $ is trivial 
in \eqref{eq:niceext}. But then $\phi$ is trivial on $B$, so $\phi = 0$.
\end{proof}

\begin{corollary}
	Let $Q$ be the maximal quotient of $\Pi$ that is a central 
	extension of $X$, and let $Z = \ker(Q \to X)$ be the abelian pro-$p$ group giving the extension. Then 
	$$
		Z = \Pi_1/\Pi_2 \cong X \wedge_{\zp} X.
	$$
\end{corollary}

\begin{proof} Inflation provides an injection from  $\mathrm{Hom}(Z,\mathbb{Q}_p/\mathbb{Z}_p)$ to $\mathrm{Hom}(B,\mathbb{Q}_p/\mathbb{Z}_p)^X$. It is an isomorphism because the kernel of an element of $\mathrm{Hom}(B,\mathbb{Q}_p/\mathbb{Z}_p)^X$ defines a central extension of $X$.  The corollary
now follows upon taking the Pontryagin dual of the isomorphism in (\ref{eq:thetanicer}).
\end{proof}

\section{Central self-adjoint extensions}
\label{s:maximal}

We continue with the notation of Sections \ref{s:CM} and \ref{s:maintheorems}, supposing that $n = 2$ 
and that $\ell = \rank_{\Lambda} X_{\mc{S}}^{\psi} = 2$.  This is equivalent to saying we have two CM types $\mc{S}_1$ and $ \mc{S}_2$
with the property that when $\mc{S} = \mc{S}_1 \cup \mc{S}_2$, the sum of the local degrees of the primes in $\mc{T}_1 = \mc{S}-\mc{S}_1$
is $2$, and the same is true for $\mc{T}_2 = \mc{S}-\mc{S}_2$.
We let $K_{\mc{S}}^{(p)}$ be the maximal pro-$p$ extension of $K$ inside the maximal
$\mc{S}$-ramified extension $K_\mc{S}$ of $K$.  Set $G_{K,\mc{S}} = \mathrm{Gal}(K_\mc{S}/K)$, $\Pi = \Gal(K_{\mc{S}}^{(p)}/K)$, and 
let $L_i$ denote the fixed field of $\Pi_i$ for $i \ge 1$.  In particular, using our previous notation, $L_1 = L$ is the maximal abelian pro-$p$
extension of $K$ which is unramified outside of $\mc{S}$ and  $X_{\mc{S}} = \Gal(L/K) = \Pi^{\ab}$.

The conjugation action of
$\tilde{\Pi} = \Gal(K_{\mc{S}}^{(p)}/E)$ on $\Pi$ gives a subgroup $\Phi$ of $\mr{Aut}(\Pi)$ to which one can apply 
Lemma \ref{lem:Tadjoint}, as in Remark \ref{ex:niceex}. 
The resulting pairing 
$$\langle \ , \ \rangle_{\Phi} \colon X_{\mc{S}} \times X_{\mc{S}} \to (\Pi_1/\Pi_2)_{\Phi,\mathrm{s}}$$
on $\Pi^{\ab}$ is the projection of the commutator pairing to the maximal quotient of $\Pi_1/\Pi_2$ for which it becomes self-adjoint with 
respect to the $\tilde{\Pi}$-action.  

The actions of $\tilde{\Pi}$ 
on $\Pi^{\ab}$ and on $\Pi_1/\Pi_2$ factor through $\Gal(K/E) = \mc{G} = \Delta \times \Gamma$,
where $\Delta = \Gal(F/E)$ is finite, abelian and of order prime to $p$ and $\Gamma = \zp^r$.
That is, $\Pi^{\ab} = X_\mc{S}$ and $\Pi_1/\Pi_2$ are modules for the group
ring $\Omega = \zp\ps{\mc{G}}$.  
 
The following lemma is clear.

\begin{lemma}
\label{lem:ourdef}
The kernel of the natural homomorphism $\Pi_1 \to (\Pi_1/\Pi_2)_{\Phi,\mathrm{s}}$
is $\Gal(K_{\mc{S}}^{(p)}/N)$, where $N$ is the maximal extension of $L$ inside
$K_{\mc{S}}^{(p)}$ having the following properties:  
\begin{enumerate}
\item[(i)] $N$ is Galois over $E$, and
$\Gal(N/L)$ is central in $\Gal(N/K)$;
\item[(ii)] the
commutator pairing
$$X_{\mc{S}}  \times X_{\mc{S}} = \Gal(L/K) \times \Gal(L/K) \to \Gal(N/L)$$
resulting from (i) is alternating and self-adjoint with respect to the action of
$\mc{G}$ by conjugation on $X_\mc{S}$.
\end{enumerate}
\end{lemma}

We also need the following consequence of weak Leopoldt, which we prove for more general sets $\mc{S}$.

\begin{lemma}
\label{lem:BigTarBaby} For any subset $\mc{S}$ of $S_f$ containing a CM type, the group $\rH^2(G_{K,\mc{S}},\qp/\zp)$ is trivial.
\end{lemma}

\begin{proof}  
First, we recall that the weak Leopoldt conjecture implies the statement in the case of $S_f$.
That is, \cite[Props. 3 and 4]{Greenberg89} imply that $\rH^2(\mathrm{Gal}(K_{S_f}/F'E^{\mathrm{cyc}}), \qp/\zp)=0$ for any number
field $F'$ in $K_{S_f}$.  Since $E^{\mr{cyc}} \subset K$, we then need only take the direct limit over all finite extensions $F'$ of $F$ contained in $K$ to see that $\rH^2(G_{K,S_f},\qp/\zp) = 0$.

Given this, the exact sequence of base terms of the Hochschild-Serre spectral sequence arising from the exact sequence
$$
	1 \to \Gal(K_{S_f}/K_{\mc{S}}) \to G_{K,S_f} \to G_{K,\mc{S}} \to 1
$$
yields an exact sequence
\begin{equation}
\label{eq:secondseq}
\rH^1(G_{K,S_f},\qp/\zp) \xrightarrow{\Res} \rH^1(\Gal(K_{S_f}/K_{\mc{S}}),\qp/\zp)^{G_{K,\mathcal{S}}} \to \rH^2(G_{K,\mathcal{S}}, \qp/\zp) \to 0.
\end{equation}
Thus, it will suffice to show that the restriction map $\Res$ is surjective.  

Setting $G = G_{K,\mc{S}}$ to shorten notation
and letting $J$ denote the maximal abelian pro-$p$ quotient of $\Gal(K_{S_f}/K_{\mc{S}})$, 
the Pontryagin dual of $\Res$ is the map on Galois groups
$$
	J_G \to X_{S_f}
$$
from the $G$-coinvariant group of $J$ to the $p$-ramified Iwasawa module over $K$. It then suffices to see that this map is injective.

By definition, $J$ is generated by its inertia groups at places of $K_{\mc{S}}$ over $\mc{S}^c$.  By the usual transitivity of the
Galois action on places, any two decomposition groups at primes over the same prime of $K$ become identified in the coinvariant group $J_G$.  In particular, we may speak of the inertia group $T_w$ of $J_G$ at a prime $w$ of $K$ lying over a prime in $\mc{S}^c$.

As any such $w$ is unramified in $K_{\mc{S}}/K$, any decomposition group in $G$ at a place over $w$ is procyclic.  Let $N$ be the subfield of $K_{S_f}$ which is the fixed field of the kernel of the natural surjection $\mathrm{Gal}(K_{S_f}/K_{\mathcal{S}}) \to J_G$.    We have an exact sequence
$$
	1 \to J_G \to \Gal(N/K) \to G \to 1.
$$
Consequently, any decomposition group in $\Gal(N/K)$ at a place over $w$ is a central extension of a procyclic
group by an abelian group and is therefore itself abelian. In particular, $T_w$ is a quotient of the
inertia group $\mf{I}_w$ in the Galois group of the maximal abelian pro-$p$ extension of the completion $K_w$.

The product of all $\mf{I}_w$ over primes $w$ lying over primes in $\mc{S}^c$ can be identified with $I_{\mc{S}^c}$ of \eqref{eq:DI}. Since $J_G$ is generated by its inertia groups $T_w$, we obtain a surjective map $I_{\mc{S}^c} \to J_G$.  Composing
this with $J_G \to X_{S_f}$, it remains only to show that $I_{\mc{S}^c} \to X_{S_f}$ is injective.  This follows from the injectivity in Lemma \ref{lem:ranklemma}, since $\mc{S}$ contains a CM type and the character $\psi$ therein was arbitrary.
\end{proof}

Because of Lemma \ref{lem:BigTarBaby},  $\theta_{X_{\mc{S}}}$ of \eqref{eq:thetanice} is an isomorphism by Lemma \ref{lem:isom} applied with $G = G_{K,\mc{S}}$.
Dually, we then have canonical isomorphisms
$$
	\Gal(L_2/L) = \Pi_1/\Pi_2 \cong X_{\mc{S}} \wedge_{\zp} X_{\mc{S}}.
$$

\begin{remark}
Since $X_{\mc{S}}$ is rank two over $\Omega$, and $\Omega$ is free of infinite rank over $\zp$, 
the (completed) wedge product $X_{\mc{S}} \wedge_{\zp} X_{\mc{S}}$
is not finitely generated over $\Omega$.  Thus $\Gal(L_2/L)$
is by Lemma \ref{lem:isom} also not finitely generated over $\Omega$.  In other words,
the second graded quotient in the lower central series of the maximal pro-$p$
quotient of $G_{K,{\mc{S}}}$ is too big for us to readily attach to it invariants arising
from finitely generated $\Omega$-modules. We remedy this by taking (completed) wedge products over $\Omega$ 
and considering the associated quotients of $\Gal(L_2/L)$.
\end{remark}

\begin{remark}
\label{rem:functorialstuff}  Suppose $M$ is a profinite abelian group with a continuous action of $\Omega = \mathbb{Z}_p\ps{\mathcal{G}}$.
The completed wedge product $M \wedge_{\mathbb{Z}_p} M$ is the topological completion of the
usual wedge product of $M$ with itself as a $\mathbb{Z}_p$-module, and there is a universal continuous alternating
bilinear $\mathbb{Z}_p$-module map $M \times M \to M \wedge_{\mathbb{Z}_p} M$.  Similarly,
$M \wedge_{\Omega} M$ is the topological completion of the usual wedge product, and there is a universal
continuous alternating bilinear $\Omega$-module map $M \times M \to M \wedge_{\Omega} M$.  This
implies that $M \wedge_{\Omega} M$ is the quotient of $M \wedge_{\mathbb{Z}_p} M$ by the closure
of the subgroup generated by all elements of the form $gm_1 \wedge m_2 - m_1 \wedge gm_2$
with $g \in \mathcal{G}$ and $m_1, m_2 \in M$.
\end{remark}

From this point forward, we use the notation $N$ for the field $N$ of Lemma \ref{lem:ourdef}.
Recall that by the $\psi$-isotypical component of a compact $\Omega$-module $M$, we mean
$M^{\psi} = M \cotimes{\zp[\Delta]} W$ for the map $\zp[\Delta] \to W$ induced by $\psi$.

\begin{proposition}
\label{prop:theresult}  Suppose that $n = \ell = 2$.
\begin{enumerate}
\item[(i)] The commutator pairing induces an isomorphism $X_{\mc{S}} \wedge_{\Omega} X_{\mc{S}} \xrightarrow{\sim} \Gal(N/L)$.
\item[(ii)] Under the isomorphism in (i), the action of $g \in \mc{G} = \Gal(K/E)$ on $\Gal(N/L)$
by conjugation corresponds to the action of $g^2$ on $X_\mc{S} \wedge_{\Omega} X_{\mc{S}}$
which sends $v_1 \wedge v_2$ to $g^2 v_1 \wedge v_2 = g v_1 \wedge gv_2$.
\item[(iii)] 
The $\psi$-isotypical component of $X_{\mc{S}} \wedge_{\Omega} X_{\mc{S}}$ 
is isomorphic to $X_{\mc{S}}^\psi \wedge_{\Omega_W} X_{\mc{S}}^\psi$, where $\Omega_W=W\hat{\otimes}_{\zp}\Omega$.  
\end{enumerate}
\end{proposition}

\begin{proof}
An element
 $h \in \Hom(X_{\mc{S}} \wedge_{\zp} X_{\mc{S}},
\qp/\zp) = \Hom(\Gal(L_2/L), \qp/\zp)$ lies in the subgroup
$\Hom(X_{\mc{S}} \wedge_{\Omega} X_{\mc{S}},
\qp/\zp)$ if and only if
$$h(gx_1 \wedge x_2) = h(x_1 \wedge gx_2)$$
for all $g \in \mc{G}$ and $x_1, x_2 \in X_{\mc{S}}$, so if and only if $h$ is self-adjoint for the action of $\mc{G}$.  
In view of the definitions of $L_2$ and $N$, this shows (i).

For (ii), note that the commutator
pairing is equivariant with respect to conjugation.  Thus if $g \in \mc{G} = \Gal(K/E)$,
$v_1, v_2 \in X_{\mc{S}}$ and $v_1 \wedge v_2 \in X_{\mc{S}} \wedge_{\Omega} X_{\mc{S}} = \Gal(N/L)$, the conjugate $\tilde{g} (v_1 \wedge v_2) \tilde{g}^{-1}$ of
$v_1 \wedge v_2$ by a lift $\tilde{g}$ of $g$ to $\Gal(N/E)$ equals
$g(v_1) \wedge g(v_2)$.  Since the commutator pairing is $\mc{G}$-adjoint when we take its values in $\Gal(N/L)$, we find
$$g(v_1) \wedge g(v_2) = g^2(v_1) \wedge v_2.$$

For part (iii), we have 
$$
	W \hat{\otimes}_{\mathbb{Z}_p} X_{\mathcal{S}} = \bigoplus_{\psi} X^{\psi}_{\mathcal{S}}
$$ 
where the sum is over the characters $\psi \colon \Delta \to W^{\times}$.  
For $i \in \{1, 2\}$, let $\psi_i \in  \mathrm{Hom}(\Delta,W^{\times})$ and $v_i \in X_{\mc{S}}^{\psi_i}$.  
The action of $\sigma \in \Delta$ on the element $v_1 \wedge v_2$ of $X_{\mc{S}}^{\psi_1} \wedge_{\Omega_W} X_{\mc{S}}^{\psi_2}$ is given by both
\begin{eqnarray*}(\sigma v_1) \wedge v_2 = \psi_1(\sigma)(v_1 \wedge v_2) & \mr{and} &  v_1 \wedge (\sigma v_2) = \psi_2(\sigma)(v_1 \wedge v_2).
\end{eqnarray*}
Thus $v_1 \wedge v_2 = 0$ if $\psi_1 \ne \psi_2$, and $X_{\mc{S}}^{\psi} \wedge_{\Omega_W} X_{\mc{S}}^{\psi}$  is the $\psi$-isotypical component of $X_{\mc{S}} \wedge_{\Omega} X_{\mc{S}}$.
By Remark \ref{rem:functorialstuff}, the canonical surjection
$$
	\mu \colon X_{\mc{S}}^\psi \wedge_{\Lambda} X_{\mc{S}}^\psi \to X_{\mc{S}}^\psi \wedge_{\Omega_W} X_{\mc{S}}^\psi
$$ 
is an homomorphism of $\Lambda$-modules which
identifies $ X_{\mc{S}}^\psi \wedge_{\Omega_W} X_{\mc{S}}^\psi $ with the quotient of $X_{\mc{S}}^\psi \wedge_{\Lambda} X_{\mc{S}}^\psi$
by the closure of the subgroup generated by all elements of the form $gv \wedge v' - v\wedge gv'$ with $g \in \mathcal{G}$ and $v, v' \in X_{\mc{S}}^\psi$.  However, $\mathcal{G} = \Delta \times \Gamma$,
and all such elements are zero both for $g \in \Delta$ and for $g \in \Gamma$, so we conclude $\mu$ is an isomorphism.    
\end{proof}

\begin{remark} 
\label{rem:result}
Phrased differently, part (ii) of Proposition \ref{prop:theresult} says that the action of $g \in \mc{G}$ on $X_\mc{S} \wedge_{\Omega} X_{\mc{S}}$
given by $g(v_1 \wedge v_2) = g(v_1) \wedge v_2 = v_1 \wedge g(v_2)$ for $v_1, v_2 \in X_{\mc{S}}$ is identified via part (i) with a canonical square root for the action of $g$ by conjugation on $\Gal(N/L)$.  Part (iii) tells us that $X_{\mc{S}}^\psi \wedge_{\Lambda} X_{\mc{S}}^\psi$ is identified with the $\psi$-isotypical component of $\Gal(N/L)$ with respect to this square root action.
\end{remark}

Let $P$ be one of $\mc{T}_1$ or $\mc{T}_2$.  We need to characterize
the image of $\bigwedge^2_{\Omega} I_P$ in $\bigwedge^2_{\Omega} X_{\mc{S}}$,
for $I_P$ associated to inertia groups at the primes over those in $P$, as defined in \eqref{eq:DI}.

\begin{proposition}
\label{prop:unramified}  
Let $N_P$ be the maximal extension of $L$ inside $N$
such that all the inertia subgroups in $\Gal(N_P/K)$
of primes over $P$
in $N_P$ are abelian. 
Under the map induced by the commutator pairing, the cokernel of the map
$$
	I_P \wedge_{\Omega} I_P \to X_{\mc{S}} \wedge_{\Omega} X_{\mc{S}}
$$
induced by the canonical map $I_P \to X_{\mc{S}}$ is identified with $\Gal(N_P/L)$.
\end{proposition}

\begin{proof}
We show that the kernel of the restriction map
$$\Hom(X_{\mc{S}} \wedge_{\Omega} X_{\mc{S}},\qp/\zp) \to
 \Hom(I_{P} \wedge_{\Omega} I_{P},\qp/\zp)
$$
is 
$\Hom(\Gal(N_P/L),\qp/\zp)$.
Let $f \in \Hom(B,\qp/\zp)^{X_{\mc{S}}}$  determine
$$
	h = \theta_{X_{\mc{S}}}\circ \Tra(f) \in \Hom(X_{\mc{S}} \wedge_\Omega X_{\mc{S}},\qp/\zp)
$$
via the isomorphism \eqref{eq:thetanice}.  We must determine when $h$ has
trivial restriction to  $\Hom(I_{P} \wedge_\Omega I_{P},\qp/\zp)$.
The interpretation of $h$ as a commutator pairing says that this will be the case if and only
if inside the central extension $G_{K,{\mc{S}}}/\ker(f)$ of $X_{\mc{S}} = G_{K,{\mc{S}}}/B$
by $B/\ker(f)$, the inverse image $\tilde{I}_{P}$ in $G_{K,{\mc{S}}}/\ker(f)$ of the image of $I_P$ in $X_{\mc{S}}$ is abelian.  The subgroup $I^\diamond_{P}$ of $\tilde{I}_{P}$ generated by inertia groups of primes over $P$ surjects onto
$I_{P}$.  So since $G_{K,{\mc{S}}}/\ker(f)$ is a central extension of $X_{\mc{S}}$ by
$B/\ker(f)$, the commutators of any two elements of $\tilde{I}_{P}$ will be trivial
if and only if the same is true of $I^\diamond_{P}$. Thus the condition that $h$ has trivial
restriction to  $\Hom(I_{P} \wedge_\Omega I_{P},\qp/\zp)$ is
the same as requiring that $I^\diamond_{P}$ is abelian.
\end{proof}

Define $M_P/L$ to be the maximal subextension of 
$N/L$ such that $M_P/L$ is unramified at all primes of $M_P$ over ${P}$.
One has $M_P \subset N_P$ because the inertia groups in $\Gal(M_P/K)$
at primes over $P$ inject into inertia groups of primes over $P$
in the abelian group $X_{\mc{S}} = \Gal(L/K)$, hence are themselves abelian.  On the other hand, $N_P/L$
need not be unramified at primes over $P$, so $N_P$ may be a nontrivial extension of $M_P$.
The following lemma shows that this makes no difference from the point of view of second Chern classes.

\begin{lemma}
\label{lem:RalphLemma}
The kernel of the surjection $\Gal(N_P/L) \to \Gal(M_P/L)$
is supported in codimension at least $3$.
\end{lemma}

\begin{proof}
Since $K \subset L \subset M_P \subset N_P \subset N$ and $\Gal(N/K)$
is finitely generated as an $\Omega$-module, the group $\Gal(N_P/M_P)$ is finitely
generated as an $\Omega$-module.  Since $M_P$ is the maximal extension of $L$ in $N$ that is unramified over $P$, 
it is equal to $(N_P)^{J_P}$ for $J_P$ the subgroup of
$\Gal(N_P/L)$ generated by the inertia groups of primes of $N_P$ over $P$.
Thus $\Gal(N_P/M_P)$ is generated as an $\Omega$-module by finitely many inertia 
subgroups $J_{\mf{Q}}$ of $\Gal(N_P/L)$ for primes $\mf{Q}$ over $P$ in $N_P$.

Let $\mf{p} \in P$, and let $\mf{Q}$ be a prime of $N_P$ above $\mf{p}$.  
By Lemma \ref{lem:sweepup}(ii) and the definition of $N_P$,
the completion of $N_P$ at $\mf{Q}$ is contained in the maximal abelian pro-$p$ extension $K_{\mf{p}}^{\ab,(p)}$ of the 
completion $K_{\mf{p}}$ of $K$ at the prime under $\mf{Q}$.   
Since $M_P/L$ is completely split at all primes over $\mf{p}$, the completions of $M_P$ and $L$ at primes
under $\mf{Q}$ are equal.
Thus $J_{\mf{Q}}$ is a quotient of the Galois group 
$H_{\mf{p}}$ of $K_{\mf{p}}^{\ab,(p)}$ over the completion of $L$ at the prime under $\mf{Q}$.  
Since the $J_{\mf{Q}}$ for  $\mf{Q}$ over $\mf{p} \in P$ generate $\Gal(N_P/M_P)$ as an $\Omega$-module, this implies that
$\Gal(N_P/M_P)$ is a quotient of the $\Omega$-submodule of $I_P$ given by
\begin{equation}
\label{eq:blah}
 \bigoplus_{\mf{p} \in P} \Omega \cotimes{\zp\ps{\mc{G}_{\mf{p}}}} H_{\mf{p}}.
\end{equation}

The $\psi$-isotypical component of 
\eqref{eq:blah} is contained in the kernel of the homomorphism $I_{P}^{\psi} \to Y_{\mc{S}}^{\psi}$
since this homomorphism factors through the injection $X_{\mc{S}}^\psi \to Y_{\mc{S}}^\psi$.
By Proposition \ref{prop:localseq}, Remark \ref{EKp} and Lemma \ref{lem:sweepup}, the homomorphism $I_{P}^{\psi} \to (I_{P}^{\psi})^{**} $ is injective.
The localization at codimension two primes of the map $(I_{P}^{\psi})^{**} \to (Y_{\mc{S}}^\psi)^{**}$
is a map between free modules of the same rank which has torsion cokernel and is therefore injective.
Thus, the kernel of $I_{P}^{\psi} \to Y_{\mc{S}}^\psi$ must be supported in codimension at least three. 
\end{proof}

Set $U = \Gal(N/L)$ and $V = \Gal(M/L)$, where $M = M_{\mc{T}_1} \cap M_{\mc{T}_2}$.
We denote by $U^{\sqrt{\psi}}$ (resp. $V^{\sqrt{\psi}}$)
the $\psi$-isotypical component of $U$ (resp. $V$) with respect
to the square root action of the conjugation action described in Remark
\ref{rem:result}.
The following is the main theorem of this section.  It contains  Theorem D of the introduction.

\begin{theorem}  
\label{thm:mainquartic}
With the assumptions and notations of Theorem \ref{thm:unconditional}, there is an isomorphism
$\bigwedge_{\Omega}^2 X_{\mc{S}}^\psi \xrightarrow{\sim} U^{\sqrt{\psi}}$ induced by the commutator pairing on $X_{\mc{S}}$. This yields surjections  
\begin{equation}
\label{eq:wedgitthree}
 \frac{\bigwedge_{\Omega}^2 X_{\mc{S}}^\psi}{\im(\bigwedge_{\Omega}^2  I_{\mc{T}_1}^\psi) + \im(\bigwedge_{\Omega}^2  I_{\mc{T}_2 }^\psi)}
\twoheadrightarrow V^{\sqrt {\psi}} \quad \mbox{and} \quad  \frac{(\bigwedge_{\Omega}^2 X_{\mc{S}}^\psi)_{\tf}}{(\bigwedge_{\Omega}^2  I_{\mc{T}_1}^\psi)_{\tf} + (\bigwedge_{\Omega}^2  I_{\mc{T}_2}^\psi)_{\tf}} \twoheadrightarrow \frac{V^{\sqrt{\psi}}}{\im(\Tor(U^{\sqrt{\psi}}))} 
\end{equation}
whose kernels are supported in codimension at least $3$. (Here, we use ``$\im$'' to denote the not necessarily isomorphic image of a module
under a canonical map.)
Moreover, we have a congruence of second Chern classes
\begin{multline}
\label{eq:c2formulaquarticupdate}
c_2	\left(\frac{\Lambda}{(\mc{L}_{\mc{S}_1,\psi}/\theta,\mc{L}_{\mc{S}_2,\psi}/\theta)} \right)  \equiv
t_2\left( \frac{V^{\sqrt{\psi}}}{\im(\Tor(U^{\sqrt{\psi}}))} \right) \\
+ c_2\left(\frac{\theta}{\theta_0} \cdot \frac{\Lambda}{\Fitt(\rE^2(X_{\mc{S}^c}^{\omega\psi^{-1} })(1))}\right) \bmod \mc{Z}_{\mc{S},\psi}.
\end{multline}
\end{theorem}

\begin{proof}
The isomorphism $\bigwedge_{\Omega}^2 X_{\mc{S}}^\psi = U^{\sqrt{\psi}}$ results from Proposition \ref{prop:theresult}  and Remark \ref{rem:result}.
By Proposition \ref{prop:unramified}, we have an identification
\begin{equation}
\label{eq:wedgit}
 \frac{\bigwedge^2_{\Omega} X_{\mc{S}}}{\bigwedge^2_{\Omega}  I_{\mc{T}_i}}
\cong \Gal(N_{\mc{T}_i}/L)
\end{equation}
of $\Omega$-modules.
Proposition \ref{prop:theresult} further identifies the $\psi$-isotypical component of the left-hand side of \eqref{eq:wedgit} with $\psi$-isotypical component of the right-hand side for the square root of the conjugation action on $\Gal(N_{\mc{T}_i}/L)$.
From \eqref{eq:wedgit}, we get an isomorphism
\begin{equation}
\label{eq:wedgittwo}
 \frac{\bigwedge_{\Omega}^2 X_{\mc{S}}}{\bigwedge_{\Omega}^2  I_{\mc{T}_1} + \bigwedge_{\Omega}^2  I_{\mc{T}_2 }}
\cong \Gal((N_{\mc{T}_1} \cap N_{\mc{T}_2})/L).
\end{equation}
By Lemma \ref{lem:RalphLemma}, $\Gal\left ((N_{\mc{T}_1} \cap N_{\mc{T}_2})/(M_{\mc{T}_1} \cap M_{\mc{T}_2})\right )$ is supported in codimension at least 3 as a module for $\Omega$ so (\ref{eq:wedgittwo}) gives (\ref{eq:wedgitthree}).    
Substituting these facts into Theorem \ref{thm:unconditional}, we obtain Theorem \ref{thm:mainquartic}.
\end{proof}

\footnotesize
\noindent\textit{Acknowledgments.}
The authors would like to thank G. Pappas for helpful discussions and T. Kataoka for noticing a mistake in an earlier version of the proof of Lemma \ref{lem:sweepup}. They also thank the referees for comments and suggestions which helped to improve the article. F. Bleher was partially supported by NSF FRG Grant No.\ DMS-1360621 and NSF Grant No.\ DMS-1801328. T. Chinburg was partially supported by  NSF FRG Grant No.\ DMS-1360767, NSF SaTC Grants No.\ CNS-1513671/1701785,  and Simons Foundation Grant No.\ 338379. R. Greenberg was partially supported by NSF FRG Grant No.\ DMS-1360902. R. Sharifi was partially supported by NSF FRG Grant No.\ DMS-1360583 and NSF Grant No.\ DMS-1801963.

\vspace{-2ex}

\appendix

\section{Field Diagram}
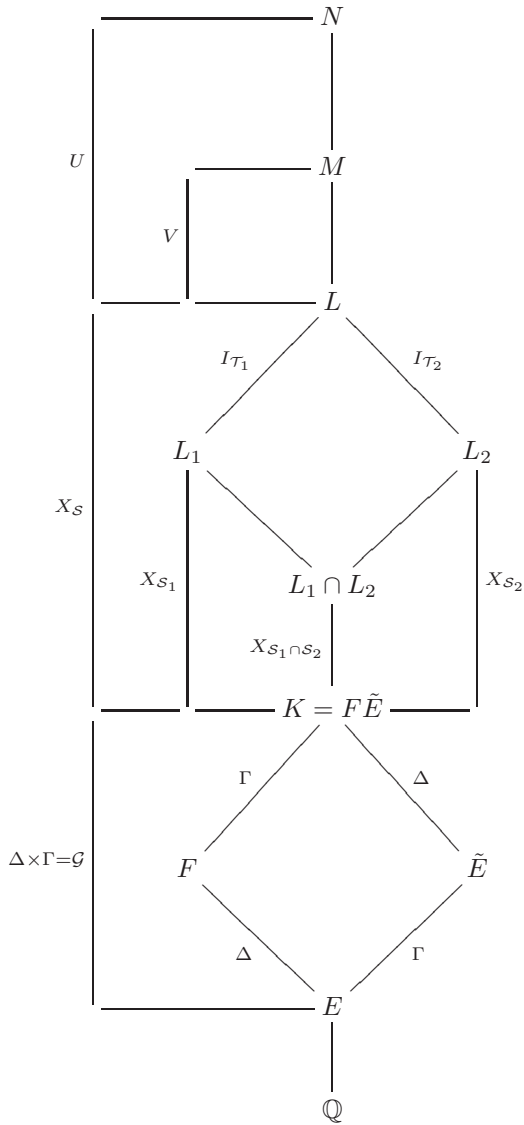
\label{s:FieldDiagram}
$$\xymatrix{
\ar@{-}[rr]\ar@{-}[dd]_{U}&&N\ar@{-}[d]&&{\begin{minipage}[t]{2.75in} {$N$ = maximal abelian pro-$p$ extension of $L$ unramified outside $\mathcal{S}=\mathcal{S}_1\cup\mathcal{S}_2$ satisfying the conditions of Definition \ref{def:Zdef},}\end{minipage}}
\\
&\ar@{-}[r]\ar@{-}[d]_{V}&M\ar@{-}[d]&&{\begin{minipage}[t]{2.75in} {$M$ = maximal subextension of $N$ containing $L$ such that $M/L$ is unramified outside $\mathcal{S}_1\cap\mathcal{S}_2$,}\end{minipage}}
\\
\ar@{-}[r]\ar@{-}[ddd]_{X_\mathcal{S}}&\ar@{-}[r]&L\ar@{-}[ld]_{I_{\mathcal{T}_1}}\ar@{-}[rd]^{I_{\mathcal{T}_2}}&
&{\begin{minipage}[t]{2.75in} {$L$ = maximal abelian pro-$p$ extension of $K$ unramified outside $\mathcal{S}=\mathcal{S}_1\cup\mathcal{S}_2$,}\end{minipage}}
\\
&L_1\ar@{-}[dd]_{X_{\mc{S}_1}}\ar@{-}[rd]&&L_2\ar@{-}[dd]^{X_{\mathcal{S}_2}}\ar@{-}[ld]
&{\begin{minipage}[t]{2.75in} {$L_i$ = maximal abelian pro $p$-extension of $K$ unramified outside $\mathcal{S}_i$, where 
$\mathcal{S}_1$ and $ \mathcal{S}_2$ are distinct $p$-adic CM types,}\end{minipage}}
\\
&&L_1\cap L_2\ar@{-}[d]_{X_{\mathcal{S}_1\cap \mathcal{S}_2}}\\
\ar@{-}[r]\ar@{-}[dd]_{\Delta\times\Gamma=\mathcal{G}}&\ar@{-}[r]&K=F\tilde{E}\ar@{-}[ld]_{\Gamma}\ar@{-}[rd]^{\Delta}&\ar@{-}[l]&{\begin{minipage}[t]{2.75in} {$\Gamma=\mathrm{Gal}(K/F)\cong \mathrm{Gal}(\tilde{E}/E)\cong \mathbb{Z}_p^r$, $r\ge d+1$,\\ 
$\Delta=\mathrm{Gal}(F/E)\cong \mathrm{Gal}(K/\tilde{E})$,}\end{minipage}}\\
&F\ar@{-}[rd]_{\Delta}&&\tilde{E}\ar@{-}[ld]^{\Gamma}
&{\begin{minipage}[t]{2.75in} {$\tilde{E}=$ compositum of all $\mathbb{Z}_p$-extensions of $E$, \\ 
$F=$ finite abelian extension of $E$ of degree prime to $p$ containing the $p$th roots of unity,}\end{minipage}}
\\
\ar@{-}[rr]&&E\ar@{-}[d]&
&{\begin{minipage}[t]{2.75in} {$E=$ CM field of degree $2d$.}\end{minipage}}
\\
&&\mathbb{Q}&
&
}$$

\section{Notation Index}
\label{s:NoteIndex}

\begin{tabular}[t]{lp{5in}}
$\rT_n(M)$&maximal submodule of $M$ supported in codimension at least $n$\\
$t_n(M)$ & $n$th Chern class of $\rT_n(M)$\\
$c_n(M)$& $n$th Chern class of $\rT_n(M)$ if $M=\rT_n(M)$\\
$E$ &a CM field of degree $2d$\\
$E^+$ &the maximal totally real subfield of $E$\\
$F$ &a finite Galois extension of $E$ of degree prime to $p$ (with varying extra hypotheses)
\end{tabular}
\newpage \noindent
\begin{tabular}[t]{lp{5in}}
$K$&compositum of $F$ with all $\mathbb{Z}_p$-extensions of $E$\\
$\Delta$&$\mathrm{Gal}(F/E)$\\
$\Gamma$&$\mathrm{Gal}(K/F)$\\
$\mathcal{G}$&$\mathrm{Gal}(K/E)$\\
$\chi_p$& the $p$-adic cyclotomic character of $\mc{G}$ with values in $\zp^{\times}$\\
$\omega$&the Teichm\"uller character of $\Delta$\\
$W$&the Witt vectors over $\overline{\F}_p$\\
$\psi$&a character of $\Delta$ valued in $W^{\times}$\\
$\Lambda$&$W\ps{\Gamma}$\\
$\Omega $&$ \zp\ps{\mc{G}}$\\
$A^{\psi}$&the $\psi$-isotypical component of $A \cotimes{\zp} W$ for a compact $\Omega$-module $A$\\
$S_f$& the set of primes of $E$ over $p$\\
$\Sigma$&a subset of $S_f$\\
$X_\Sigma$&the $\Sigma$-ramified Iwasawa module over $K$\\
$\mr{L}_{\Sigma,\psi}$&the power series defining the Katz $p$-adic $L$-function of $\psi$ when $\Sigma$ is a CM type\\
$\mc{L}_{\Sigma,\psi}$&a choice of generator for the characteristic ideal of $X_\Sigma^\psi$ when $\Sigma$ is a CM type\\
$\mc{S}$&a union of ($p$-adic) CM types $\mc{S}_i$\\
$\overline {\mathcal{S}}$&the image under complex conjugation of $\mc{S}$\\
$\mc{S}^c$& the set of primes over $p$ not in $\mc{S}$\\
$\ell$&the $\Lambda$-rank of $X_{\mc{S}}^{\psi}$ (with a related abstract usage in Section \ref{sec:exterior})\\
$A(1)$&the Tate twist of a compact $\Lambda$-module $A$ by the cyclotomic character of $\Gamma$\\
$A^\iota$& a compact $\Lambda$-module $A$ in which the $\Gamma$-action is inverted\\
$\rE^i(A)$&$ \Ext^i_{\Lambda}(A^{\iota},\Lambda)$ for a finitely generated $\Lambda$-module $A$ (also with $\Omega$ in
place of $\Lambda$)\\
$A_{\tf}$&$A/\rT_1(A)$ for a finitely generated module $A$ over an integral domain\\
$\bigwedge^{\ell}A$&the $\ell$th exterior power of a finitely generated module $A$\\
$\Fitt(A)$& the $0$th Fitting ideal of a finitely generated module $A$\\
$V^{\sqrt{\psi}}$&the $\psi$-isotypic component of $V$ for the canonical square root action of $\mathcal{G}$ on $V$\\
$K_{\mathcal{S}}^{(p)}$&the maximal $\mathcal{S}$-ramified
pro-$p$ extension of $K$\\
$S$&the set $S_{p,\infty}$ of all places of $E$ over $p$ and $\infty$\\
$G_{F',S}$ & $\Gal(F'_S/F')$ for $F'_S$ the maximal $S$-ramified extension of $F'/F$\\
$\mathcal{Q}$&$ \mathrm{Gal}(F_S/E)$\\
$\rC_{\Iw}(K,T)$&the Iwasawa cochain complex of  a compact $\zp\ps{\mc{Q}}$-module $T$\\
$\RGa_{\Iw}(K,T)$&the class of  $\rC_{\Iw}(K,T)$  in the derived category \\
$\rH^i_{\Iw}(K,T)$& the $i$th cohomology group of  $\rC_{\Iw}(K,T)$\\
$\rC_{v,\Iw}(K,T)$&the local Iwasawa cochain complex of  $T$ at $v \in S_f$\\
$M^*$&$\rE^0(M)$ for a finitely generated $\Omega$-module $M$\\
$\rC_{\Sigma,\Iw}(K,T)$&$\Cone\left(\rC_{\Iw}(K,T) \to \bigoplus_{v \in \Sigma} \rC_{v,\Iw}(K,T)\right)[-1]$ for $\Sigma \subset S_f$\\
$\rH^i_{\Sigma,\Iw}(K,T)$& the $i$th cohomology group of $\rC_{\Sigma,\Iw}(K,T)$\\
$X^\flat_\Sigma$& the maximal quotient of $X_\Sigma$ that is completely split at the primes in $S_f-\Sigma$\\
$Y_\Sigma$&$ \rH^2_{\Sigma,\Iw}(K,\zp(1))$\\
$\mc{G}_{\mf{p}}$&the decomposition group in $\mc{G}$  at a place over the prime $\mf{p}$ in $K$\\
$\mc{K}_{\mf{p}}$&$ \zp\ps{\mc{G}/\mc{G}_{\mf{p}}}$\\
$\mc{K}_\Sigma$&$ \bigoplus_{\mf{p} \in \Sigma} \mc{K}_{\mf{p}}$\\
$\mc{K}_{\Sigma,0}$&the kernel of the augmentation map $\mc{K}_\Sigma \to \zp$\\
$\Gamma_{\mf{p}}$&the decomposition group $\mc{G}_{\mf{p}}\cap \Gamma$ in $\Gamma$
\end{tabular}
\newpage \noindent
\begin{tabular}[t]{lp{5in}}
$r_{\mf{p}}$&$\rank_{\zp} \Gamma_{\mf{p}}$\\
$\mf{D}_{\mf{p}}$&the Galois group of the maximal abelian pro-$p$ extension of the completion of $K$ over $\mf{p}$\\
$\mf{I}_{\mf{p}}$&the inertia subgroup of $\mf{D}_{\mf{p}}$\\
$D_{\mf{p}}$&$ \Omega \cotimes{\zp\ps{\mc{G}_{\mf{p}}}} \mf{D}_{\mf{p}}$\\
$I_{\mf{p}}$&$ \Omega \cotimes{\zp\ps{\mc{G}_{\mf{p}}}} \mf{I}_{\mf{p}}$\\
$D_\Sigma$&$\bigoplus_{\mf{p} \in \Sigma} D_{\mf{p}}$\\
$I_\Sigma$&$ \bigoplus_{\mf{p} \in \Sigma} I_{\mf{p}}$\\
$d_{\mf{p}}$&$[E_{\mf{p}}:\Q_p]$\\
$Z_{\Sigma}$&$\zp$ if $\Sigma = S_f$ and $r \ge 2$, zero otherwise\\
$d_{\Sigma}$&$\sum_{\mf{p} \in \Sigma} d_{\mf{p}}$\\
$\rQ(M)$&$R/\Fitt(M)$ for a finitely generated module $M$ over an integral domain $R$\\
$\mc{T}$& the union of $\mc{T}_i = \mc{S} - \mc{S}_i$ for the CM types $\mc{S}_i$ with union $\mc{S}$ (except in
Lemma \ref{lem:ranklemma})\\
$\mc{U}_{\mf{p},\psi}$&the set of codimension two primes of $\Lambda$ in the support of $\mc{K}_{\mf{p}}^{\psi}$ \\
$\overline{\mc{U}}_{\mf{p},\psi}$&the set of codimension two primes of $\Lambda$ in the support of $(\mc{K}_{\mf{p}}^{\omega\psi^{-1}})^{\iota}(1)$\\
$\mc{U}_{\Sigma,\psi}$&$\bigcup_{\mf{p} \in \Sigma} \ \mc{U}_{\mf{p},\psi}$\\
$\overline{\mc{U}}_{\Sigma,\psi} $&$\bigcup_{\mf{p} \in \Sigma} \ \overline{\mc{U}}_{\mf{p},\psi}$\\
$\mc{Z}_{\Sigma,\psi}$&the free abelian group on $\mc{V}_{\Sigma,\psi} =  \mc{U}_{\Sigma^c,\psi} \cup \overline{\mc{U}}_{\Sigma,\psi}$\\
$\Pi_i$&the closure of the $i$th term in the lower central series of a profinite group $\Pi$\\
$\Phi$& a subgroup of  the group $\mr{Aut}(\Pi)$ \\
$(\Pi_1/\Pi_2)_{\Phi,\mathrm{s}}$&the maximal quotient of $\Pi_1/\Pi_2$ by a $\Phi$-stable subgroup with self-adjoint commutator
\end{tabular}

\end{document}